\documentclass[draft]{article}

\usepackage{amsmath,amsfonts,amsthm,amssymb,amscd,enumerate}

\usepackage{amentashortcuts2}
\usepackage{amentastyle2}
\usepackage[draft=false, bookmarks, colorlinks=true]{hyperref}
\usepackage{authblk}

\title{Tent spaces over metric measure spaces under doubling and related assumptions}
\author{Alex Amenta\thanks{amenta@fastmail.fm: supported by the Australian Research Council Discovery grant DP120103692.}}
\affil{Mathematical Sciences Institute, Australian National University}

\makeatletter

\def\@maketitle{%
  \newpage
  \null
  \vskip 2em%
  \begin{center}%
  \let \footnote \thanks
    {\LARGE \@title \par}%
    \vskip 1.5em%
    {\large
      \lineskip .5em%
      \begin{tabular}[t]{c}%
        \@author
      \end{tabular}\par}%
  \end{center}%
  \par
  \vskip 1.5em}

\makeatother

\begin{document}
\maketitle

\begin{abstract}
In this article, we define the Coifman--Meyer--Stein tent spaces $T^{p,q,\ga}(X)$ associated with an arbitrary metric measure space $(X,d,\gm)$ under minimal geometric assumptions.
While gradually strengthening our geometric assumptions, we prove duality, interpolation, and change of aperture theorems for the tent spaces.
Because of the inherent technicalities in dealing with abstract metric measure spaces, most proofs are presented in full detail.
\end{abstract}

\tableofcontents

\section{Introduction}

The purpose of this article is to indicate how the theory of tent spaces, as developed by Coifman, Meyer, and Stein for Euclidean space in \cite{CMS85}, can be extended to more general metric measure spaces.
Let $X$ denote the metric measure space under consideration.
If $X$ is doubling, then the methods of \cite{CMS85} seem at first to carry over without much modification.
However, there are some technicalities to be considered, even in this context.
This is already apparent in the proof of the atomic decomposition given in \cite{eR07}.

Further still, there is an issue with the proof of the main interpolation result of \cite{CMS85} (see Remark \ref{cmserror} below).
Alternate proofs of the interpolation result have since appeared in the literature --- see for example \cite{HTV91}, \cite{aB92}, \cite{CV00}, and \cite{mK11} --- but these proofs are given in the Euclidean context, and no indication is given of their general applicability.
In fact, the methods of \cite{HTV91} and \cite{aB92} can be used to obtain a partial interpolation result under weaker assumptions than doubling.
This result relies on some tent space duality; we show in Section \ref{dvvi} that this holds once we assume that the uncentred Hardy--Littlewood maximal operator is of strong type $(r,r)$ for all $r > 1$.\footnote{This fact is already implicit in \cite{CMS85}.}

Finally, we consider the problem of proving the change of aperture result when $X$ is doubling.
The proof in \cite{CMS85} implicitly uses a geometric property of $X$ which we term (NI), or `nice intersections'.
This property is independent of doubling, but holds for many doubling spaces which appear in applications --- in particular, all complete Riemannian manifolds have `nice intersections'.
We provide a proof which does not require this assumption.

\addtocontents{toc}{\protect\setcounter{tocdepth}{1}}
\subsection*{Acknowledgements}
We thank Pierre Portal and Pascal Auscher for their comments and suggestions, particularly regarding the proofs of Lemmas \ref{cptest} and \ref{measurability}.
We further thank Lashi Bandara, Li Chen, Mikko Kemppainen and Yi Huang for discussions on this work, as well as the participants of the Workshop in Harmonic Analysis and Geometry at the Australian National University for their interest and suggestions.
Finally, we thank the referee for their detailed comments.
\addtocontents{toc}{\protect\setcounter{tocdepth}{2}}

\section{Spatial assumptions}\label{assumptions}

Throughout this article, we implicitly assume that $(X,d,\gm)$ is a metric measure space; that is, $(X,d)$ is a metric space and $\gm$ is a Borel measure on $X$.
The ball centred at $x \in X$ of radius $r > 0$ is the set
\begin{equation*}
	B(x,r) := \{y \in X : d(x,y) < r\},
\end{equation*}
and we write $V(x,r) := \gm(B(x,r))$ for the volume of this set.
We assume that the volume function $V(x,r)$ is finite\footnote{Since $X$ is a metric space, this implies that $\gm$ is $\gs$-finite.} and positive; one can show that $V$ is automatically measurable on $X \times \RR_+$.

There are four geometric assumptions which we isolate for future reference:

\begin{description}
\item[(Proper)]a subset $S \subset X$ is compact if and only if it is both closed and bounded, and the volume function $V(x,r)$ is lower semicontinuous as a function of $(x,r)$;\footnote{Note that this is a strengthening of the usual definition of a proper metric space, as the usual definition does not involve a measure.
We have abused notation by using the word `proper' in this way, as it is convenient in this context.}

\item[(HL)] the uncentred Hardy--Littlewood maximal operator $\mc{M}$, defined for measurable functions $f$ on $X$ by
\begin{equation}\label{uHL}
	\mc{M}(f)(x) := \sup_{B \ni x} \frac{1}{\gm(B)} \int_B |f(y)| \, d\gm(y)
\end{equation}
where the supremum is taken over all balls $B$ containing $x$, is of strong type $(r,r)$ for all $r > 1$;

\item[(Doubling)]
there exists a constant $C > 0$ such that for all $x \in X$ and $r > 0$,
\begin{equation*}
	V(x,2r) \leq CV(x,r);
\end{equation*}

\item[(NI)]
for all $\ga,\gb > 0$ there exists a positive constant $c_{\ga,\gb} > 0$ such that for all $r > 0$ and for all $x, y \in X$ with $d(x,y) < \ga r$,
\begin{equation*}
	\frac{\gm(B(x,\ga r) \cap B(y,\gb r))}{V(x, \ga r)} \geq c_{\ga,\gb}.
\end{equation*}

\end{description}

We do not assume that $X$ satisfies any of these assumptions unless mentioned otherwise.
However, readers are advised to take $(X,d,\gm)$ to be a complete Riemannian manifold with its geodesic distance and Riemannian volume if they are not interested in such technicalities.

It is well-known that doubling implies (HL).
However, the converse is not true. 
See for example \cite{FSSU02} and \cite{pS83}, where it is shown that (HL) is true for $\RR^2$ with the Gaussian measure.
We will only consider (NI) along with doubling, so we remark that doubling does not imply (NI): one can see this by taking $\RR^2$ (now with Lebesgue measure) and removing an open strip.\footnote{One could instead remove an open bounded region with sufficiently regular boundary, for example an open square. This yields a connected example.}
One can show that all complete doubling length spaces---in particular, all complete doubling Riemannian manifolds---satisfy (NI).

\section{The basic tent space theory}

\subsection{Initial definitions and consequences}\label{initdefns}

Let $X^+$ denote the `upper half-space' $X \times \RR_+$, equipped with the product measure $d\gm(y) \, dt/t$ and the product topology.
Since $X$ and $\RR_+$ are metric spaces, with $\RR_+$ separable, the Borel $\gs$-algebra on $X^+$ is equal to the product of the Borel $\gs$-algebras on $X$ and $\RR_+$, and so the product measure on $X^+$ is Borel (see \cite[Lemma 6.4.2(i)]{vB07II}).

We say that a subset $C \subset X^+$ is \emph{cylindrical} if it is contained in a cylinder: that is, if there exists $x \in X$ and $a,b,r > 0$ such that $C \subset B(x,r) \times (a,b)$.
Note that cylindricity is equivalent to boundedness when $X^+$ is equipped with an appropriate metric, and that compact subsets of $X^+$ are cylindrical.

Cones and tents are defined as usual: for each $x \in X$ and $\ga > 0$, the \emph{cone of aperture $\ga$ with vertex $x$} is the set
\begin{equation*}
	\gG^\ga(x) := \{(y,t) \in X^+ : y \in B(x,\ga t)\}.
\end{equation*}
For any subset $F \subset X$ we write
\begin{equation*}
	\gG^\ga(F) := \bigcup_{x \in F} \gG^\ga(x).
\end{equation*}
For any subset $O \subset X$, the \emph{tent of aperture $\ga$ over $O$} is defined to be the set
\begin{equation*}
	T^\ga(O) := (\gG^\ga(O^c))^c.
\end{equation*}
Writing
\begin{equation*}
  F_O(y,t) := \frac{\dist(y,O^c)}{t} = t^{-1} \inf_{x \in O^c} d(y,x),
\end{equation*}
one can check that $T^\ga(O) = F_O^{-1}((\ga,\infty))$.
Since $F_O$ is continuous (due to the continuity of $\dist(\cdot,O^c)$), we find that tents over open sets are measurable, and so it follows that cones over closed sets are also measurable.
We remark that tents (resp. cones) over non-open (resp. non-closed) sets may not be measurable.

Let $F \subset X$ be such that $O := F^c$ has finite measure.
Given $\gg \in (0,1)$, we say that a point $x \in X$ has \emph{global $\gg$-density with respect to $F$} if for all balls $B$ containing $x$,
\begin{equation*}
	\frac{\gm(B \cap F)}{\gm(B)} \geq \gg.
\end{equation*}
We denote the set of all such points by $F_\gg^*$, and define $O_\gg^* := (F_\gg^*)^c$.
An important fact here is the equality
\begin{equation*}
	O_\gg^* = \{x \in X : \mc{M}(\mb{1}_O)(x) > 1-\gg\},
\end{equation*}
where $\mb{1}_O$ is the indicator function of $O$.
We emphasise that $\mc{M}$ denotes the \emph{uncentred} maximal operator.
When $O$ is open (i.e. when $F$ is closed), this shows that $O \subset O_\gg^*$ and hence that $F_\gg^* \subset F$.
Furthermore, the function $\mc{M}(\mb{1}_O)$ is lower semicontinuous whenever $\mb{1}_O$ is locally integrable (which is always true, since we assumed $O$ has finite measure), which implies that $F_\gg^*$ is closed (hence measurable) and that $O_\gg^*$ is open (hence also measurable).
Note that if $X$ is doubling, then since $\mc{M}$ is of weak-type $(1,1)$, we have that
\begin{equation*}
	\gm(O_\gg^*) \lesssim_{\gg,X} \gm(O).
\end{equation*}

\begin{rmk}
	In our definition of points of $\gg$-density, we used balls containing $x$ rather than balls centred at $x$ (as is usually done).
	This is done in order to avoid using the centred maximal function, which may not be measurable without assuming continuity of the volume function $V(x,r)$.
\end{rmk}

Here we find it convenient to introduce the notion of the \emph{$\ga$-shadow} of a subset of $X^+$.
For a subset $C \subset X^+$, we define the $\ga$-shadow of $C$ to be the set
\begin{equation*}
	S^\ga(C) := \{x \in X : \gG^\ga(x) \cap C \neq \varnothing\}.
\end{equation*}
Shadows are always open, for if $A \subset X^+$ is any subset, and if $x \in S^\ga(A)$, then there exists a point $(z,t_z) \in \gG^\ga(x) \cap A$, and one can easily show that $B(x,\ga t_z - d(x,z))$ is contained in $S^\ga(A)$.

The starting point of the tent space theory is the definition of the operators $\mc{A}_q^\ga$ and $\mc{C}_q^\ga$.
For $q \in (0,\infty)$, the former is usually defined for measurable functions $f$ on $\RR_+^{n+1}$ (with values in $\RR$ or $\CC$, depending on context) by
\begin{equation*}
	\mc{A}_q^\ga(f)(x)^q := \iint_{\gG^\ga(x)} |f(y,t)|^q \, \frac{d\gl(y) \, dt}{t^{n+1}}
\end{equation*}
where $x \in \RR^n$ and $\gl$ is Lebesgue measure.
There are four reasonable ways to generalise this definition to our possibly non-doubling metric measure space $X$:\footnote{We do not claim that these are the only reasonable generalisations.} these take the form
\begin{equation*}
	\mc{A}_q^\ga(f)(x)^q := \iint_{\gG^\ga(x)} |f(y,t)|^q \, \frac{d\gm(y)}{V(\mb{a},\mb{b}t)}\, \frac{dt}{t}
\end{equation*}
where $\mb{a} \in \{x,y\}$ and $\mb{b} \in \{1,\ga\}$.
In all of these definitions, if a function $f$ on $X^+$ is supported on a subset $C \subset X^+$, then $\mc{A}_q^\ga(f)$ is supported on $S^\ga(C)$; we will use this fact repeatedly in what follows.
Measurability of $\mc{A}_q^\ga(f)(x)$ in $x$ when $\mb{a} = y$ follows from Lemma \ref{measurability} in the Appendix; the choice $\mb{a} = x$ can be taken care of with a straightforward modification of this lemma.
The choice $\mb{a} = x$, $\mb{b} = 1$ appears in \cite{AMR08,eR07}, and the choice $\mb{a} = y$, $\mb{b} = 1$ appears in \cite[\textsection 3]{MvNP11}.
These definitions all lead to equivalent tent spaces when $X$ is doubling.
We will take $\mb{a} = y$, $\mb{b} = \ga$ in our definition, as it leads to the following fundamental technique, which works with no geometric assumptions on $X$.

\begin{lem}[Averaging trick]\label{avgtrick}
	Let $\ga > 0$, and suppose $\gF$ is a nonnegative measurable function on $X^+$.
	Then
	\begin{equation*}
		\int_X \iint_{\gG^\ga(x)} \gF(y,t) \, \frac{d\gm(y)}{V(y,\ga t)} \, \frac{dt}{t} \, d\gm(x) = \iint_{X^+} \gF(y,t) \, d\gm(y) \, \frac{dt}{t}.
	\end{equation*}
\end{lem}

\begin{proof}
	This is a straightforward application of Fubini--Tonelli's theorem, which we present explicitly due to its importance in what follows:
	\begin{align*}
		\int_X \iint_{\gG^\ga(x)} \gF(y,t) \, \frac{d\gm(y)}{V(y,\ga t)} \frac{dt}{t} \, d\gm(x)
		&= \int_X \int_0^\infty \int_X \mb{1}_{B(x,\ga t)}(y) \gF(y,t) \, \frac{d\gm(y)}{V(y,\ga t)} \, \frac{dt}{t} \, d\gm(x) \\
		&= \int_0^\infty \int_X \int_X \mb{1}_{B(y,\ga t)}(x) \, d\gm(x) \, \gF(y,t) \, \frac{d\gm(y)}{V(y,\ga t)} \, \frac{dt}{t} \\
		&= \int_0^\infty \int_X \frac{V(y,\ga t)}{V(y,\ga t)} \gF(y,t) \, d\gm(y) \, \frac{dt}{t} \\
		&= \iint_{X^+} \gF(y,t) \, d\gm(y) \, \frac{dt}{t}.
	\end{align*}
\end{proof}

We will also need the following lemma in order to prove that our tent spaces are complete.
Here we need to make some geometric assumptions.

\begin{lem}\label{cptest}
	Let $X$ be proper or doubling.
	Let $p,q,\ga > 0$, let $K \subset X^+$ be cylindrical, and suppose $f$ is a measurable function on $X^+$.
	Then
	\begin{equation}\label{cptesteq}
		\norm{\mc{A}_q^\ga(\mb{1}_K f)}_{L^p(X)} \lesssim \norm{f}_{L^q(K)} \lesssim \norm{\mc{A}_q^\ga(f)}_{L^p(X)},
	\end{equation}
	with implicit constants depending on $p$, $q$, $\ga$, and $K$.
\end{lem}

\begin{proof}
  Write
  \begin{equation*}
    K \subset B(x,r) \times (a,b) =: C
  \end{equation*}
  for some $x \in X$ and $a,b,r > 0$.
  We claim that there exist constants $c_0,c_1 > 0$ such that for all $(y,t) \in C$,
  \begin{equation*}
    c_0 \leq V(y,\ga t) \leq c_1.
  \end{equation*}
  If $X$ is proper, this is an immediate consequence of the lower semicontinuity of the ball volume function (recall that we are assuming this whenever we assume $X$ is proper) and the compactness of the closed cylinder $\overline{B(x,r)} \times [a,b]$.
  If $X$ is doubling, then we argue as follows.
  Since $V(y,\ga t)$ is increasing in $t$, we have that
  \begin{equation*}
    \min_{(y,t) \in C} V(y,\ga t) \geq \min_{y \in B(x,r)} V(y,\ga a)
  \end{equation*}
  and
  \begin{equation*}
    \max_{(y,t) \in C} V(y,\ga t) \leq \max_{y \in B(x,r)} V(y,\ga b).
  \end{equation*}
  By the argument in the proof of Lemma \ref{betas} (in particular, by \eqref{infimum}), there exists $c_0 > 0$ such that
  \begin{equation*}
    \min_{y \in B(x,r)} V(y,\ga a) \geq c_0.
  \end{equation*}
  Furthermore, since
  \begin{equation*}
    V(y,\ga b) \leq V(x, \ga b + r)
  \end{equation*}
  for all $y \in B(x,r)$, we have that
  \begin{equation*}
    \max_{y \in B(x,r)} V(y,\ga b) \leq V(x, \ga b + r) =: c_1,
  \end{equation*}
  proving the claim.

  To prove the first estimate of \eqref{cptesteq}, write
  \begin{align*}
    \norm{\mc{A}_q^\ga(\mb{1}_K f)}_{L^p(X)}
    &= \left(\int_{S^\ga(K)} \left( \iint_{\gG^\ga(x)} \mb{1}_K(y,t)|f(y,t)|^q \, \frac{d\gm(y)}{V(y,\ga t)} \, \frac{dt}{t} \right)^{\frac{p}{q}} \, d\gm(x) \right)^{\frac{1}{p}} \\
    &\lesssim_{c_0,q} \left( \int_{S^\ga(K)} \left( \iint_K |f(y,t)|^q \, d\gm(y) \, \frac{dt}{t} \right)^{\frac{p}{q}} \, d\gm(x) \right)^{\frac{1}{p}} \\
    &\lesssim_K \norm{f}_{L^q(K)}.
  \end{align*}

  To prove the second estimate, first choose finitely many points $(x_n)_{n=1}^N$ such that
  \begin{equation*}
    \overline{B(x,r)} \subset \bigcup_{n=1}^N B(x_n,\ga a/2)
  \end{equation*}
  using either compactness of $\overline{B(x,r)}$ (in the proper case) or doubling.\footnote{In the doubling case, this is a consequence of what is usually called `geometric doubling'. A proof that this follows from the doubling condition can be found in \cite[\textsection III.1]{CW71}. }
  We then have
  \begin{align*}
    \left( \iint_K |f(y,t)|^q \, d\gm(y) \, \frac{dt}{t} \right)^{\frac{1}{q}}
      &\lesssim_{c_1} \left( \iint_K \sum_{n=1}^N \mb{1}_{B(x_n,\ga a/2)}(y) |f(y,t)|^q \, \frac{d\gm(y)}{V(y,\ga t)} \, \frac{dt}{t} \right)^{\frac{1}{q}} \\
      &\lesssim_{X,q} \sum_{n=1}^N \left( \iint_{K} \mb{1}_{B(x_n,\ga a/2)}(y) |f(y,t)|^q \, \frac{d\gm(y)}{V(y,\ga t)} \, \frac{dt}{t} \right)^{\frac{1}{q}}.
  \end{align*}
  If $x,y \in B(x_n,\ga a/2)$, then $d(x,y) < \ga a < \ga t$ (since $t > a$), and so
  \begin{equation}\label{containment}
    \iint_K \mb{1}_{B(x_n,\ga a/2)}(y) |f(y,t)|^q \, \frac{d\gm(y)}{V(y,\ga t)} \, \frac{dt}{t}
    \leq \iint_{\gG^\ga(x)} |f(y,t)|^q \, \frac{d\gm(y)}{V(y,\ga t)} \, \frac{dt}{t}.
  \end{equation}
  When $p \geq q$, we use H\"older's inequality along with \eqref{containment} to write 
  \begin{align*}
    &\sum_{n=1}^N \left( \iint_K \mb{1}_{B(x_n,\ga a/2)}(y) |f(y,t)|^q \, \frac{d\gm(y)}{V(y,\ga t)} \, \frac{dt}{t} \right)^{\frac{1}{q}} \\
    &= \sum_{n=1}^N \left( \frac{1}{V(x_n,\ga a/2)} \int_{B(x_n,\ga a/2)} \iint_K \mb{1}_{B(x_n,\ga a/2)}(y) |f(y,t)|^q \, \frac{d\gm(y)}{V(y,\ga t)} \, \frac{dt}{t} \, d\gm(x) \right)^{\frac{1}{q}} \\
    &\leq \sum_{n=1}^N \left( \frac{1}{V(x_n,\ga a/2)} \int_{B(x_n,\ga a/2)} \left( \iint_K \mb{1}_{B(x_n,\ga a/2)}(y) |f(y,t)|^q \, \frac{d\gm(y)}{V(y,\ga t)} \, \frac{dt}{t}\right)^{\frac{p}{q}} d\gm(x) \right)^{\frac{1}{p}} \\
    &\leq \sum_{n=1}^N \left( \frac{1}{V(x_n,\ga a/2)} \int_{B(x_n,\ga a/2)} \left(\iint_{\gG^\ga(x)} |f(y,t)|^q \, \frac{d\gm(y)}{V(y,\ga t)} \, \frac{dt}{t}\right)^{\frac{p}{q}} d\gm(x) \right)^{\frac{1}{p}} \\
    &\lesssim_{K,p} \norm{\mc{A}_q^\ga(f)}_{L^p(X)},
  \end{align*}
  completing the proof in this case.
  When  $p < q$, the situtaion can be handled using Minkowski's inequality as follows.
  Using $p/q < 1$, we have
  \begin{align*}
    &\left( \frac{1}{V(x_n,\ga a/2)} \int_{B(x_n,\ga a/2)} \iint_K \mb{1}_{B(x_n,\ga a/2)}(y) |f(y,t)|^q \, \frac{d\gm(y)}{V(y,\ga t)} \, \frac{dt}{t} \, d\gm(x) \right)^{\frac{1}{q}} \\
    &\leq \left( \frac{1}{V(x_n,\ga a/2)} \left( \int_{B(x_n,\ga a/2)} \left( \iint_{K} \mb{1}_{B(x_n,\ga a/2)}(y) |f(y,t)|^q \, \frac{d\gm(y)}{V(y,\ga t)} \, \frac{dt}{t} \right)^{\frac{p}{q}} d\gm(x) \right)^{\frac{q}{p}} \right)^{\frac{1}{q}} \\
    &\leq C\left( \frac{1}{V(x_n,\ga a/2)} \int_{B(x_n,\ga a/2)} \left( \iint_K \mb{1}_{B(x_n,\ga a/2)}(y) |f(y,t)|^q \, \frac{d\gm(y)}{V(y,\ga t)} \, \frac{dt}{t} \right)^{\frac{p}{q}} d\gm(x) \right)^{\frac{1}{p}},
  \end{align*}
  where
  \begin{equation*}
    C = C(p,q,\ga,K) = \max_{n} \left(V(x_n,\ga a/2)^{\frac{1}{p}-\frac{1}{q}}\right).
  \end{equation*}
  We can then proceed as in the case where $p \geq q$.
\end{proof}

As usual, with $\ga > 0$ and $p,q \in (0,\infty)$, we define the tent space (quasi-)norm of a measurable function $f$ on $X^+$ by
\begin{equation*}
	\norm{f}_{T^{p,q,\ga}(X)} := \norm{\mc{A}_q^\ga(f)}_{L^p(X)},
\end{equation*}
and the tent space $T^{p,q,\ga}(X)$ to be the (quasi-)normed vector space consisting of all such $f$ (defined almost everywhere) for which this quantity is finite.

\begin{rmk}
  One can define the tent space as either a real or complex vector space, according to one's own preference.
  We will implicitly work in the complex setting (so our functions will always be $\CC$-valued).
  Apart from complex interpolation, which demands that we consider complex Banach spaces, the difference is immaterial.
\end{rmk}

\begin{prop}\label{completenessetc}
	Let $X$ be proper or doubling.
	For all $p,q,\ga \in (0,\infty)$, the tent space $T^{p,q,\ga}(X)$ is complete and contains $L_c^q(X^+)$ (the space of functions $f \in L^q(X^+)$ with cylindrical support) as a dense subspace.
\end{prop}

\begin{proof}
	Let $(f_n)_{n \in \NN}$ be a Cauchy sequence in $T^{p,q,\ga}(X)$.
	Then by Lemma \ref{cptest}, for every cylindrical subset $K \subset X^+$ the sequence $(\mb{1}_K f_n)_{n \in \NN}$ is Cauchy in $L^q(K)$.
	We thus obtain a limit
	\begin{equation*}
		f_K := \lim_{n \to \infty} \mb{1}_K f_n \in L^q(K)
	\end{equation*}
	for each $K$.
	If $K_1$ and $K_2$ are two cylindrical subsets of $X^+$, then $f_{K_1}|_{K_1 \cap K_2} = f_{K_2}|_{K_1 \cap K_2}$, so by making use of an increasing sequence $\{K_m\}_{m \in \NN}$ of cylindrical subsets of $X^+$ whose union is $X^+$ (for example, we could take $K_m := B(x,m) \times (1/m,m)$ for some $x \in X$) we obtain a function $f \in L_\text{loc}^q(X^+)$ with $f|_{K_m} = f_{K_m}$ for each $m \in \NN$.\footnote{We interpret `locally integrable on $X^+$' as meaning `integrable on all cylinders', rather than `integrable on all compact sets'.}
	This is our candidate limit for the sequence $(f_n)_{n \in \NN}$.
	
	To see that $f$ lies in $T^{p,q,\ga}(X)$, write for any $m,n \in \NN$
	\begin{align*}
		\norm{\mb{1}_{K_m}f}_{T^{p,q,\ga}(X)} &\lesssim_{p,q} \norm{\mb{1}_{K_m}(f-f_n)}_{T^{p,q,\ga}(X)} + \norm{\mb{1}_{K_m} f_n}_{T^{p,q,\ga}(X)} \\
		&\leq C_{p,q,\ga,X,m}\norm{f-f_n}_{L^q(K_m)} + \norm{f_n}_{T^{p,q,\ga}(X)},
	\end{align*}
	the $(p,q)$-dependence in the first estimate being relevant only for $p < 1$ or $q < 1$, and the second estimate coming from Lemma \ref{cptest}.
	Since the sequence $(f_n)_{n \in \NN}$ converges to $\mb{1}_{K_m} f$ in $L^q(K_m)$ and is Cauchy in $T^{p,q,\ga}(X)$, we have that
	\begin{equation*}
		\norm{\mb{1}_{K_m} f}_{T^{p,q,\ga}(X)} \lesssim \sup_{n \in \NN} \norm{f_n}_{T^{p,q,\ga}(X)}
	\end{equation*}
	uniformly in $m$.
	Hence $\norm{f}_{T^{p,q,\ga}(X)}$ is finite.
	
	We now claim that for all $\ge > 0$ there exists $m \in \NN$ such that for all sufficiently large $n \in \NN$, we have
	\begin{equation*}
		\norm{\mb{1}_{K_m^c}(f_n-f)}_{T^{p,q,\ga}(X)} \leq \ge.
	\end{equation*}
	Indeed, since the sequence $(f_n)_{n \in \NN}$ is Cauchy in $T^{p,q,\ga}(X)$, there exists $N \in \NN$ such that for all $n,n^\prime \geq N$ we have $\norm{f_n - f_{n^\prime}}_{T^{p,q,\ga}(X)} < \ge/2$.
	Furthermore, since
	\begin{equation*}
		\lim_{m \to \infty} \norm{\mb{1}_{K_m^c}(f_N - f)}_{T^{p,q,\ga}(X)} = 0
	\end{equation*}
	by the Dominated Convergence Theorem, we can choose $m$ such that
	\begin{equation*}
		\norm{\mb{1}_{K_m^c}(f_N - f)}_{T^{p,q,\ga}(X)} < \ge/2.
	\end{equation*}
	Then for all $n \geq N$,
	\begin{align*}
		\norm{\mb{1}_{K_m^c}(f_n - f)}_{T^{p,q,\ga}(X)}
		&\lesssim_{p,q} \norm{\mb{1}_{K_m^c}(f_n - f_N)}_{T^{p,q,\ga}(X)} + \norm{\mb{1}_{K_m^c}(f_N - f)}_{T^{p,q,\ga}(X)} \\
		&\leq \norm{f_n - f_N}_{T^{p,q,\ga}(X)} + \norm{\mb{1}_{K_m^c}(f_N - f)}_{T^{p,q,\ga}(X)} \\
		&< \ge,
	\end{align*}
	proving the claim.
	
	Finally, by the previous remark, for all $\ge > 0$ we can find $m$ such that for all sufficiently large $n \in \NN$ we have
	\begin{align*}
		\norm{f_n - f}_{T^{p,q,\ga}(X)} &\lesssim_{p,q} \norm{\mb{1}_{K_m}(f_n - f)}_{T^{p,q,\ga}(X)} + \norm{\mb{1}_{K_m^c} (f_n - f)}_{T^{p,q,\ga}(X)} \\
		&< \norm{\mb{1}_{K_m}(f_n - f)}_{T^{p,q,\ga}(X)} + \ge \\
		&\leq C(p,q,\ga,X,m)\norm{f_n-f}_{L^q(K_m)} + \ge.
	\end{align*}
	Taking the limit of both sides as $n \to \infty$, we find that $\lim_{n \to \infty} f_n = f$ in $T^{p,q,\ga}(X)$, and therefore $T^{p,q,\ga}(X)$ is complete.
	
	To see that $L_c^q(X^+)$ is dense in $T^{p,q,\ga}(X)$, simply write $f \in T^{p,q,\ga}(X)$ as the pointwise limit
	\begin{equation*}
		f = \lim_{n \to \infty} \mb{1}_{K_n} f.
	\end{equation*}
	By the Dominated Convergence Theorem, this convergence holds in $T^{p,q,\ga}(X)$.	
\end{proof}

We note that Lemma \ref{avgtrick} implies that in the case where $p=q$, we have $T^{p,p,\ga}(X) = L^p(X^+)$ for all $\ga > 0$.


In the same way as Lemma \ref{avgtrick}, we can prove the analogue of \cite[Lemma 1]{CMS85}.

\begin{lem}[First integration lemma]\label{il1}
	For any nonnegative measurable function $\gF$ on $X^+$, with $F$ a measurable subset of $X$ and $\ga > 0$,
	\begin{equation*}
		\int_F \iint_{\gG^\ga(x)} \gF(y,t) \, d\gm(y) \, dt \, d\gm(x) \leq \iint_{\gG^\ga(F)} \gF(y,t) V(y,\ga t) \, d\gm(y) \, dt.
	\end{equation*}
\end{lem}

\begin{rmk}\label{obvcoa}
	There is one clear disadvantage of our choice of tent space norm: it is no longer clear that
	\begin{equation}\label{trivialCoA}
		\norm{\cdot}_{T^{p,q,\ga}(X)} \leq \norm{\cdot}_{T^{p,q,\gb}(X)}
	\end{equation}
	when $\ga < \gb$.
	In fact, this may not even be true for general nondoubling spaces.
	This is no great loss, since for doubling spaces we can revert to the `original' tent space norm (with $\mb{a} = x$ and $\mb{b} = 1$) at the cost of a constant depending only on $X$, and for this choice of norm \eqref{trivialCoA} is immediate.
\end{rmk}

In order to define the tent spaces $T^{\infty,q,\ga}(X)$, we need to introduce the operator $\mc{C}_q^\ga$.
For measurable functions $f$ on $X^+$, we define
\begin{equation*}
	\mc{C}_q^\ga(f)(x) := \sup_{B \ni x} \left( \frac{1}{\gm(B)} \iint_{T^\ga(B)} |f(y,t)|^q \, d\gm(y) \, \frac{dt}{t} \right)^{\frac{1}{q}},
\end{equation*}	
where the supremum is taken over all balls containing $x$.
Since $\mc{C}_q^\ga(f)$ is lower semicontinuous (see Lemma \ref{inftymeas}), $\mc{C}_q^\ga(f)$ is measurable.
We define the (quasi-)norm $\norm{\cdot}_{T^{\infty,q,\ga}(X)}$ for functions $f$ on $X^+$ by
\begin{equation*}
	\norm{f}_{T^{\infty,q,\ga}(X)} := \norm{\mc{C}_q^\ga(f)}_{L^\infty(X)},
\end{equation*}
and the tent space $T^{\infty,q,\ga}(X)$ as the (quasi-)normed vector space of measurable functions $f$ on $X^+$, defined almost everywhere, for which $\norm{f}_{T^{\infty,q,\ga}(X)}$ is finite.
The proof that $T^{\infty,q,\ga}(X)$ is a (quasi-)Banach space is similar to that of Proposition \ref{completenessetc} once we have established the following analogue of Lemma \ref{cptest}.

\begin{lem}\label{inftyest}
	Let $q,\ga > 0$, let $K \subset X^+$ be cylindrical, and suppose $f$ is a measurable function on $X^+$.
	Then
	\begin{equation}\label{carlest}
		\norm{f}_{L^q(K)} \lesssim \norm{f}_{T^{\infty,q,\ga}(X)},
	\end{equation}
	with implicit constant depending only on $\ga$, $q$, and $K$ (but not otherwise on $X$).
	
	Furthermore, if $X$ is proper or doubling, then we also have
	\begin{equation*}
		\norm{\mb{1}_K f}_{T^{\infty,q,\ga}(X)} \lesssim \norm{f}_{L^q(K)},
	\end{equation*}
	again with implicit constant depending only on $\ga$, $q$, and $K$.
\end{lem}

\begin{proof}
	We use Lemma \ref{betas}.
	To prove the first estimate, for each $\ge > 0$ we can choose a ball $B_\ge$ such that $T^\ga(B_\ge) \supset K$ and $\gm(B_\ge) < \gb_1(K) + \ge$.
	Then
	\begin{align*}
		\norm{f}_{L^q(K)} &\leq \norm{\mb{1}_{T^\ga(B_\ge)} f}_{L^q(X^+)} \\
		&= \gm(B_\ge)^{\frac{1}{q}} \gm(B_\ge)^{-\frac{1}{q}} \norm{\mb{1}_{T^\ga(B_\ge)} f}_{L^q(X^+)} \\
		&\leq (\gb_1(K) + \ge)^{\frac{!}{q}} \norm{f}_{T^{\infty,q,\ga}(X)}.
	\end{align*}
	In the final line we used that $\gm(B_\ge) > 0$ to conclude that $\gm(B_\ge)^{-1/q}\norm{\mb{1}_{T^\ga(B_\ge)}f}_{L^q(X^+)}$ is less than the \emph{essential} supremum of $\mc{C}_q^\ga(f)$.
	Since $\ge > 0$ was arbitrary, we have the first estimate.
	
	For the second estimate, assuming that $X$ is proper or doubling, observe that
	\begin{align*}\
		\norm{\mb{1}_K f}_{T^{\infty,q,\ga}(X)} &\leq \sup_{B \subset X} \left( \frac{1}{\gm(B)} \iint_{T^{\ga}(B) \cap K} |f(y,t)|^q \, d\gm(y) \, \frac{dt}{t} \right)^{\frac{1}{q}} \\
		&\leq \left(\frac{1}{\gb_0(K)} \iint_{K} |f(y,t)|^q \, d\gm(y) \frac{dt}{t}\right)^{\frac{1}{q}} \\
		&= \gb_0(K)^{-\frac{1}{q}} \norm{f}_{L^q(K)},
	\end{align*}
	completing the proof.
\end{proof}

\begin{rmk}
	In this section we did not impose any geometric conditions on our space $X$ besides our standing assumptions on the measure $\gm$ and the properness assumption (in the absence of doubling).
	Thus we have defined the tent space $T^{p,q,\ga}(X)$ in considerable generality.
	However, what we have defined is a \emph{global} tent space, and so this concept may not be inherently useful when $X$ is non-doubling.
	Instead, our interest is to determine precisely where geometric assumptions are needed in the tent space theory.
\end{rmk}

\subsection{Duality, the vector-valued approach, and complex interpolation}\label{dvvi}

\subsubsection{Midpoint results}

The geometric assumption (HL) from Section \ref{assumptions} now comes into play.
For $r > 0$, we denote the H\"older conjugate of $r$ by $r^\prime := r/(r-1)$.

\begin{prop}\label{duality1}
	Suppose that $X$ is either proper or doubling, and satisfies assumption (HL).
	Then for $p,q \in (1,\infty)$ and $\ga > 0$, the pairing
	\begin{equation*}
		\langle f,g \rangle := \iint_{X^+} f(y,t) \overline{g(y,t)} \, d\gm(y) \, \frac{dt}{t} \qquad (f \in T^{p,q,\ga}(X), g \in T^{p^\prime,q^\prime,\ga}(X))
	\end{equation*}
	realises $T^{p^\prime,q^\prime,\ga}(X)$ as the Banach space dual of $T^{p,q,\ga}(X)$, up to equivalence of norms.
\end{prop}

This is proved in the same way as in \cite{CMS85}.
We provide the details in the interest of self-containment.

\begin{proof}
	We first remark that if $p = q$, the duality statement is a trivial consequence of the equality $T^{p,p,\ga}(X) = L^p(X^+)$.
	
	In general, suppose $f \in T^{p,q,\ga}(X)$ and $g \in T^{p^\prime,q^\prime,\ga}(X)$.
	Then by the averaging trick and H\"older's inequality, we have
	\begin{align}\label{di}
		|\langle f,g \rangle|
		&\leq \int_X \iint_{\gG^\ga(x)} |f(y,t)\overline{g(y,t)}| \, \frac{d\gm(y)}{V(y,\ga t)} \, \frac{dt}{t} \, d\gm(x) \notag\\
		&\leq \int_X \mc{A}_q^\ga(f)(x) \mc{A}_{q^\prime}^\ga(g)(x) \, d\gm(x) \notag\\
		&\leq \norm{f}_{T^{p,q,\ga}(X)} \norm{g}_{T^{p^\prime,q^\prime,\ga}(X)}.
	\end{align}
	Thus every $g \in T^{p^\prime,q^\prime,\ga}(X)$ induces a bounded linear functional on $T^{p,q,\ga}(X)$ via the pairing $\langle \cdot,\cdot \rangle$, and so $T^{p^\prime,q^\prime,\ga}(X) \subset (T^{p,q,\ga}(X))^*$.
	
	Conversely, suppose $\ell \in (T^{p,q,\ga}(X))^*$.
	If $K \subset X^+$ is cylindrical, then by the properness or doubling assumption, we can invoke Lemma \ref{cptest} to show that $\ell$ induces a bounded linear functional $\ell_K \in (L^q(K))^*$, which can in turn be identified with a function $g_K \in L^{q^\prime}(K)$.
	By covering $X^+$ with an increasing sequence of cylindrical subsets, we thus obtain a function $g \in L_\text{loc}^{q^\prime}(X^+)$ such that $g|_K = g_K$ for all cylindrical $K \subset X^+$.
	
	If $f \in L^q(X^+)$ is cylindrically supported, then we have
	\begin{equation}\label{dualrepn}
		\iint_{X^+} f(y,t)\overline{g(y,t)} \, d\gm(y) \, \frac{dt}{t} = \iint_{\supp f} f(y,t) \overline{g_{\supp f}(y,t)} \, d\gm(y) \, \frac{dt}{t} = \ell_{\supp f}(f) = \ell(f),
	\end{equation}
	recalling that $f \in T^{p,q,\ga}(X)$ by Lemma \ref{cptest}.
	Since the cylindrically supported $L^q(X^+)$ functions are dense in $T^{p,q,\ga}(X)$, the representation \eqref{dualrepn} of $\ell(f)$ in terms of $g$ is valid for all $f \in T^{p,q,\ga}(X)$ by dominated convergence and the inequality \eqref{di}, provided we show that $g$ is in $T^{p^\prime,q^\prime,\ga}(X)$.
	
	Now suppose $p < q$.
	We will show that $g$ lies in $T^{p^\prime,q^\prime,\ga}(X)$, thus showing directly that $(T^{p,q,\ga}(X))^*$ is contained in $T^{p^\prime,q^\prime,\ga}(X)$.
	It suffices to show this for $g_K$, where $K \subset X^+$ is an arbitrary cylindrical subset, provided we obtain an estimate which is uniform in $K$.
	We estimate
	\begin{equation*}
		\norm{g_K}_{T^{p^\prime,q^\prime,\ga}(X)}^{q^\prime} = \norm{\mc{A}_{q^\prime}^\ga(g_K)^{q^\prime}}_{L^{p^\prime/q^\prime}(X)}
	\end{equation*}
	by duality.
	Let $\gy \in L^{(p^\prime/q^\prime)^\prime}(X)$ be nonnegative, with $\norm{\gy}_{L^{(p^\prime/q^\prime)^\prime}(X)} \leq 1$.
	Then by Fubini--Tonelli's theorem,
	\begin{align*}
		\int_X \mc{A}_{q^\prime}^\ga(g_K)(x)^{q^\prime} \gy(x) \, d\gm(x)
		&= \int_X \iint_{X^+} \mb{1}_{B(y,\ga t)}(x) |g_K(y,t)|^{q^\prime} \, \frac{d\gm(y)}{V(y,\ga t)} \, \frac{dt}{t} \, \gy(x) \, d\gm(x) \\
		&= \int_0^\infty \int_X \frac{1}{V(y,\ga t)} \int_{B(y,\ga t)} \gy(x) \, d\gm(x) \, |g_K(y,t)|^{q^\prime} \, d\gm(y) \, \frac{dt}{t} \\
		&= \iint_{X^+} M_{\ga t} \gy(y) |g_K(y,t)|^{q^\prime} \, d\gm(y) \, \frac{dt}{t},
	\end{align*}
	where $M_s$ is the averaging operator defined for $y \in X$ and $s > 0$ by
	\begin{equation*}
		M_s \gy(y) := \frac{1}{V(y,s)} \int_{B(y,s)} \gy(x) \, d\gm(x).
	\end{equation*}
	Thus we can write formally
	\begin{equation}\label{dual}
		\int_X \mc{A}_{q^\prime}^\ga(g_K)(x)^{q^\prime} \gy(x) \, d\gm(x) = \langle f_\gy,g \rangle,
	\end{equation}
	where we define
	\begin{equation*}
		f_\gy(y,t) := \left\{ \begin{array}{ll} M_{\ga t}\gy(y) \overline{g_K(y,t)}^{q^\prime /2} g_K(y,t)^{(q^\prime/2)-1} & \text{when $g_K(y,t) \neq 0$,} \\ 0 & \text{when $g_K(y,t) = 0$,} \end{array} \right.
	\end{equation*}
        noting that $g_K(y,t)^{(q^\prime/2)-1}$ is not defined when $g_K(y,t) = 0$ and $q^\prime < 2$.
	However, the equality \eqref{dual} is not valid until we show that $f$ lies in $T^{p,q,\ga}(X)$.
	To this end, estimate
	\begin{align*}
		\mc{A}_q^\ga(f_\gy)
		&\leq \left( \iint_{\gG^\ga(x)} M_{\ga t}\gy(y)^q |g_K(y,t)|^{q(q^\prime-1)} \, \frac{d\gm(y)}{V(y,\ga t)} \, \frac{dt}{t} \right)^{\frac{1}{q}} \\
		&\leq \left( \iint_{\gG^\ga(x)} \mc{M}\gy(x)^q |g_K(y,t)|^{q^\prime} \, \frac{d\gm(y)}{V(y,\ga t)} \, \frac{dt}{t} \right)^{\frac{1}{q}} \\
		&= \mc{M}\gy(x) \mc{A}_{q^\prime}^\ga(g_K)(x)^{q^\prime/q}.
	\end{align*}
	Taking $r$ such that $1/p = 1/r + 1/(p^\prime/q^\prime)^\prime$ and using (HL), we then have
	\begin{align*}
		\norm{\mc{A}_q^\ga(f_\gy)}_{L^p(X)}
		&\leq \norm{(\mc{M}\gy)\mc{A}_{q^\prime}^\ga(g_K)^{q^\prime/q}}_{L^p(X)} \\
		&\leq \norm{\mc{M}\gy}_{L^{(p^\prime/q^\prime)^\prime}(X)} \norm{\mc{A}_{q^\prime}^\ga(g_K)^{q^\prime/q}}_{L^r(X)} \\
		&\lesssim_X \norm{\gy}_{L^{(p^\prime/q^\prime)^\prime}(X)} \norm{\mc{A}_{q^\prime}^\ga(g_K)}_{L^{rq^\prime/q}(X)}^{q^\prime/q} \\
		&\leq \norm{\mc{A}_{q^\prime}^\ga(g_K)}_{L^{rq^\prime/q}(X)}^{q^\prime/q}.
	\end{align*}
	One can show that $rq^\prime/q = p^\prime$, and so $f_\gy$ is in $T^{p,q,\ga}(X)$ by Lemma \ref{cptest}.
	By \eqref{dual}, taking the supremum over all $\gy$ under consideration, we can write
	\begin{align*}
		\norm{g_K}_{T^{p^\prime,q^\prime,\ga}(X)}^{q^\prime}
		&\leq \norm{\ell}\norm{f_\gy}_{T^{p,q,\ga}(X)} \\
		&\lesssim_X \norm{\ell} \norm{g_K}_{T^{p^\prime,q^\prime,\ga}(X)}^{q^\prime/q},
	\end{align*}
	and consequently, using that $\norm{g_K}_{T^{p^\prime,q^\prime,\ga}(X)} < \infty$,
	\begin{equation*}
		\norm{g_K}_{T^{p^\prime,q^\prime,\ga}(X)}\lesssim_X \norm{\ell}.
	\end{equation*}
	Since this estimate is independent of $K$, we have shown that $g \in T^{p^\prime,q^\prime,\ga}(X)$, and therefore that $(T^{p,q,\ga}(X))^*$ is contained in $T^{p^\prime,q^\prime,\ga}(X)$.
	This completes the proof when $p < q$.
	
	To prove the statement for $p > q$, it suffices to show that the tent space $T^{p^\prime,q^\prime,\ga}(X)$ is reflexive.
	Thanks to the Eberlein--\u{S}mulian theorem (see \cite[Corollary 1.6.4]{AC06}), this is equivalent to showing that every bounded sequence in $T^{p^\prime,q^\prime,\ga}(X)$ has a weakly convergent subsequence.
	
	Let $\{f_n\}_{n \in \NN}$ be a sequence in $T^{p^\prime,q^\prime,\ga}(X)$ with $\norm{f_n}_{T^{p^\prime,q^\prime,\ga}(X)} \leq 1$ for all $n \in \NN$.
	Then by Lemma \ref{cptest}, for all cylindrical $K \subset X^+$ the sequence $\{f_n\}_{n \in \NN}$ is bounded in $L^{q^\prime}(K)$, and so by reflexivity of $L^{q^\prime}(K)$ we can find a subsequence $\{f_{n_j}\}_{j \in \NN}$ which converges weakly in $L^{q^\prime}(K)$.
	We will show that this subsequence also converges weakly in $T^{p^\prime,q^\prime,\ga}(X)$.
	
	Let $\ell \in (T^{p^\prime,q^\prime,\ga}(X))^*$.
	Since $p^\prime < q^\prime$, we have already shown that there exists a function $g \in T^{p,q,\ga}(X)$ such that $\ell(f) = \langle f,g \rangle$.
	For every $\ge > 0$, we can find a cylindrical set $K_\ge \subset X^+$ such that
	\begin{equation*}
		\norm{g - \mb{1}_{K_\ge}g}_{T^{p,q,\ga}(X)} \leq \ge.
	\end{equation*}
	Thus for all $i,j \in \NN$ and for all $\ge > 0$ we have
	\begin{align*}
		\ell(f_{n_i}) - \ell(f_{n_j})
		&= \langle f_{n_i}-f_{n_j},\mb{1}_{K_\ge}g \rangle + \langle f_{n_i}-f_{n_j},g-\mb{1}_{K_\ge}g \rangle \\
		&\leq \langle f_{n_i}-f_{n_j},\mb{1}_{K_\ge}g \rangle + (\norm{f_{n_i}}_{T^{p^\prime,q^\prime,\ga}(X)} + \norm{f_{n_j}}_{T^{p^\prime,q^\prime,\ga}(X)})\norm{g-\mb{1}_{K_\ge}g}_{T^{p,q,\ga}} \\
		&\leq \langle f_{n_i}-f_{n_j},\mb{1}_{K_\ge}g \rangle + 2\ge.
	\end{align*}
	As $i,j \to \infty$, the first term on the right hand side above tends to $0$, and so we conclude that $\{f_{n_j}\}_{n \in \NN}$ converges weakly in $T^{p^\prime,q^\prime,\ga}(X)$.
	This completes the proof.
\end{proof}

\begin{rmk}
	As mentioned earlier, property (HL) is weaker than doubling, but this  is still a strong assumption.
	We note that for Proposition \ref{duality1} to hold for a given pair $(p,q)$, the uncentred Hardy--Littlewood maximal operator need only be of strong type $((p^\prime/q^\prime)^\prime,(p^\prime/q^\prime)^\prime)$.
	Since $(p^\prime/q^\prime)^\prime$ is increasing in $p$ and decreasing in $q$, the condition required on $X$ is stronger as $p \to 1$ and $q \to \infty$.
\end{rmk}

Given Proposition \ref{duality1}, we can set up the vector-valued approach to tent spaces (first considered in \cite{HTV91}) using the method of \cite{aB92}.
Fix $p \in(0,\infty)$, $q \in (1,\infty)$, and $\ga > 0$.
For simplicity of notation, write
\begin{equation*}
	L_\ga^q(X^+) := L^q\left( X^+ ; \frac{d\gm(y)}{V(y,\ga t)} \, \frac{dt}{t} \right).
\end{equation*}
We define an operator $\map{T_\ga}{T^{p,q,\ga}(X)}{L^p(X;L_\ga^q(X^+))}$ from the tent space into the $L_\ga^q(X^+)$-valued $L^p$ space on $X$ (see \cite[\textsection 2]{DU77} for vector-valued Lebesgue spaces) by setting
\begin{equation*}
	T_\ga f(x)(y,t) := f(y,t) \mb{1}_{\gG^\ga(x)}(y,t).
\end{equation*}
One can easily check that
\begin{equation*}
	\norm{T_\ga f}_{L^p(X;L_\ga^q(X^+))} = \norm{f}_{T^{p,q,\ga}(X)},
\end{equation*}
and so the tent space $T^{p,q,\ga}(X)$ can be identified with its image under $T_\ga$ in $L^p(X;L_\ga^q(X^+))$, provided that $T_\ga f$ is indeed a strongly measurable function of $x \in X$.
This can be shown for $q \in (1,\infty)$ by recourse to Pettis' measurability theorem \cite[\textsection 2.1, Theorem 2]{DU77}, which reduces the question to that of weak measurability of $T_\ga f$.
To prove weak measurability, suppose $g \in L_\ga^{q^\prime}(X)$; then
\begin{equation*}
	\langle T_\ga f(x),g \rangle = \iint_{\gG^\ga(x)} f(y,t) \overline{g(y,t)} \, \frac{d\gm(y)}{V(y, \ga t)} \, \frac{dt}{t},
\end{equation*}
which is measurable in $x$ by Lemma \ref{measurability}.
Thus $T_\ga f$ is weakly measurable, and therefore $T_\ga f$ is strongly measurable as claimed.

Now assume $p,q \in (1,\infty)$ and consider the operator $\gP_\ga$, sending $X^+$-valued functions on $X$ to $\CC$-valued functions on $X^+$, given by
\begin{equation*}
(\gP_\ga F)(y,t) := \frac{1}{V(y,\ga t)} \int_{B(y,\ga t)} F(x)(y,t) \, d\gm(x)
\end{equation*}
whenever this expression is defined.
Using the duality pairing from Proposition \ref{duality1} and the duality pairing $\langle\langle \cdot,\cdot \rangle\rangle$ for vector-valued $L^p$ spaces, for $f \in T^{p,q,\ga}(X)$ and $G \in L^{p^\prime}(X;L_\ga^{q^\prime}(X^+))$ we have
\begin{align*}
	\langle \langle T_\ga f, G \rangle \rangle
	&= \int_X \iint_{X^+} T_\ga f(x)(y,t) \overline{G(x)(y,t)} \, \frac{d\gm(y)}{V(y,\ga t)} \, \frac{dt}{t} \, d\gm(x) \\
	&= \iint_{X^+} \frac{f(y,t)}{V(y,\ga t)} \int_X \mb{1}_{B(y,\ga t)}(x) \overline{G(x)(y,t)} \, d\gm(x)\, d\gm(y) \, \frac{dt}{t} \\
	&= \iint_{X^+} f(y,t) \overline{(\gP_\ga G)(y,t)} \, d\gm(y) \, \frac{dt}{t} \\
	&= \langle \langle f, \gP_\ga G \rangle \rangle.
\end{align*}
Thus $\gP_\ga$ maps $L^{p^\prime}(X;L_\ga^{q^\prime}(X^+))$ to $T^{p^\prime,q^\prime,\ga}(X)$, by virtue of being the adjoint of $T_\ga$.
Consequently, the operator $P_\ga := T_\ga \gP_\ga$ is bounded from $L^p(X;L_\ga^q(X^+))$ to itself for $p,q \in (1,\infty)$.
A quick computation shows that $\gP_\ga T_\ga = I$, so that $P_\ga$ projects $L^p(X;L_\ga^q(X^+))$ onto $T_\ga(T^{p,q,\ga}(X))$.
This shows that $T_\ga(T^{p,q,\ga}(X))$ is a complemented subspace of $L^p(X;L_\ga^q(X^+))$.
This observation leads to the basic interpolation result for tent spaces.
Here $[\cdot,\cdot]_\gq$ denotes the complex interpolation functor (see \cite[Chapter 4]{BL76}).

\begin{prop}\label{interpolation1}
	Suppose that $X$ is either proper or doubling, and satisfies assumption (HL).
	Then for $p_0$, $p_1$, $q_0$, and $q_1$ in $(1, \infty)$, $\gq \in [0,1]$, and $\ga > 0$, we have (up to equivalence of norms)
	\begin{equation*}
		[T^{p_0,q_0,\ga}(X),T^{p_1,q_1,\ga}(X)]_\gq = T^{p,q,\ga}(X),
	\end{equation*}
	where $1/p = (1-\gq)/p_0 + \gq/p_1$ and $1/q = (1-\gq)/q_0 + \gq/q_1$.
\end{prop}

\begin{proof}
	Recall the identification
	\begin{equation*}
		T^{r,s,\ga}(X) \cong T_\ga T^{r,s,\ga}(X) \subset L^r(X;L_\ga^s(X^+))
	\end{equation*}
	for all $r \in (0,\infty)$ and $s \in (1,\infty)$.
	Since
	\begin{align*}
		[L^{p_0}(X;L_\ga^{q_0}(X^+)),L^{p_1}(X;L_\ga^{q_1}(X^+))]_\gq &= L^p(X;[L_\ga^{q_0}(X^+),L_\ga^{q_1}(X^+)]_\gq) \\
		&= L^p(X;L_\ga^q(X^+))
	\end{align*}
	applying the standard result on interpolation of complemented subspaces with common projections (see \cite[Theorem 1.17.1.1]{hT78}) yields
	\begin{align*}
		[T^{p_0,q_0,\ga}(X),T^{p_1,q_1,\ga}(X)]_\gq
		&= L^p(X;L_\ga^q(X^+)) \cap (T^{p_0,q_0,\ga}(X) + T^{p_1,q_1,\ga}(X)) \\
		&= T^{p,q,\ga}(X).
	\end{align*}
\end{proof}

\begin{rmk}
	Since \cite[Theorem 1.17.1.1]{hT78} is true for any interpolation functor (not just complex interpolation), analogues of Proposition \ref{interpolation1} hold for any interpolation functor $F$ for which the spaces $L^p(X;L_\ga^q(X^+))$ form an appropriate interpolation scale.
	In particular, Proposition \ref{interpolation1} (appropriately modified) holds for real interpolation.
\end{rmk}

\begin{rmk}\label{refsug}
  Following the first submission of this article, the anonymous referee suggested a more direct proof of Proposition \ref{interpolation1}, which avoids interpolation of complemented subspaces.
  Since $T_\ga$ acts as an isometry both from $T^{p_0,q_0,\ga}(X)$ to $L^{p_0}(X;L_\ga^{q_0}(X^+))$ and from $T^{p_1,q_1,\ga}(X)$ to $L^{p_1}(X;L_\ga^{q_1}(X^+))$, if $f \in [T^{p_0,q_0,\ga}(X),T^{p_1,q_1,\ga}(X)]_\gq$, then
  \begin{equation*}
    \norm{f}_{T^{p,q,\ga}(X)} = \norm{T_\ga f}_{L^p(X;L_\ga^q(X^+))} \leq \norm{f}_{[T^{p_0,q_0,\ga}(X),T^{p_1,q_1,\ga}(X)]_\gq}
  \end{equation*}
  due to the exactness of the complex interpolation functor (and similarly for the real interpolation functor).
  Hence $[T^{p_0,q_0,\ga}(X),T^{p_1,q_1,\ga}(X)]_\gq \subset T^{p,q,\ga}(X)$, and the reverse containment follows by duality.
  We have chosen to include both proofs for their own intrinsic interest.
\end{rmk}

\subsubsection{Endpoint results}

We now consider the tent spaces $T^{1,q,\ga}(X)$ and $T^{\infty,q,\ga}(X)$, and their relation to the rest of the tent space scale.
In this section, we prove the following duality result using the method of \cite{CMS85}.

\begin{prop}\label{duality2}
	Suppose $X$ is doubling, and let $\ga > 0$ and $q \in (1,\infty)$.
	Then the pairing $\langle \cdot, \cdot \rangle$ of Proposition \ref{duality1} realises $T^{\infty,q,\ga}(X)$ as the Banach space dual of $T^{1,q,\ga}(X)$, up to equivalence of norms.
\end{prop}

As in \cite{CMS85}, we require a small series of definitions and lemmas to prove this result.
We define \emph{truncated cones} for $x \in X$, $\ga,h > 0$ by
\begin{equation*}
	\gG_h^\ga(x) := \gG^\ga(x) \cap \{(y,t) \in X^+ : t < h\},
\end{equation*}
and corresponding Lusin operators for $q > 0$ by
\begin{equation*}
	\mc{A}_q^\ga(f|h)(x) := \left(\iint_{\gG_h^\ga(x)} |f(y,t)|^q \, \frac{d\gm(y)}{V(y,\ga t)} \, \frac{dt}{t} \right)^{\frac{1}{q}}.
\end{equation*}
One can show that $\mc{A}_q^\ga(f|h)$ is measurable in the same way as for $\mc{A}_q^\ga(f)$.

\begin{lem}\label{stopfun}
	For each measurable function $g$ on $X^+$, each $q \in [1,\infty)$, and each $M > 0$, define
	\begin{equation*}
		h_{g,q,M}^\ga(x) := \sup\{h > 0 : \mc{A}_q^\ga(g|h)(x) \leq M\mc{C}_q^\ga(g)(x)\}
	\end{equation*}
	for $x \in X$.
	If $X$ is doubling, then for sufficiently large $M$ (depending on $X$, $q$, and $\ga$), whenever $B \subset X$ is a ball of radius $r$,
	\begin{equation*}
		\gm\{x \in B : h_{g,q,M}^\ga(x) \geq r\} \gtrsim_{X,\ga} \gm(B).
	\end{equation*}
\end{lem}

\begin{proof}
	Let $B \subset X$ be a ball of radius $r$.
	Applying Lemmas \ref{trunccones} and \ref{il1}, the definition of $\mc{C}_q^\ga$, and doubling, we have
	\begin{align*}
		\int_B \mc{A}_q^\ga(g|r)(x)^q \, d\gm(x)
		&= \int_B \iint_{\gG_r^\ga(x)} \mb{1}_{T^\ga((2\ga+1)B)}(y,t) |g(y,t)|^q \, \frac{d\gm(y)}{V(y,\ga t)} \, \frac{dt}{t} \, d\gm(x) \\
		&\leq \int_B \iint_{\gG^\ga(x)} \mb{1}_{T^\ga((2\ga+1)B)}(y,t) |g(y,t)|^q \, \frac{d\gm(y)}{V(y,\ga t)} \, \frac{dt}{t} \, d\gm(x) \\
		&\leq \iint_{T^\ga((2\ga+1)B)} |g(y,t)|^q \, d\gm(y) \, \frac{dt}{t} \\
		&\leq \gm((2\ga+1)B) \inf_{x \in B} \mc{C}_q^\ga(g)(x)^q \\
		&\lesssim_{X,\ga} \gm(B) \inf_{x \in B} \mc{C}_q^\ga(g)(x)^q.
	\end{align*}
	We can estimate
	\begin{equation*}
		\int_B \mc{A}_q^\ga(g|r)(x)^q \, d\gm(x) \geq (M\inf_{x \in B}\mc{C}_q^\ga(g)(x))^q \gm\{x \in B : \mc{A}_q^\ga(g|r)(x) > M\inf_{x \in B}\mc{C}_q^\ga(g)(x)\},
	\end{equation*}
	and after rearranging and combining with the previous estimate we get
	\begin{equation*}
		M^q\left( \gm(B) - \gm\{x \in B : \mc{A}_q^\ga(g|r)(x) \leq M\inf_{x \in B} \mc{C}_q^\ga(g)(x)\} \right) \lesssim_{X,\ga} \gm(B).
	\end{equation*}
	More rearranging and straightforward estimating yields
	\begin{equation*}
		\gm\{x \in B : \mc{A}_q^\ga(g|r)(x) \leq M\mc{C}_q^\ga(g)(x) \} \geq (1-M^{-q}C_{X,\ga})\gm(B).
	\end{equation*}
	Since $h_{g,q,M}^\ga(x) \geq r$ if and only if $\mc{A}_q^\ga(g|r)(x) \leq M\mc{C}_q^\ga(g)(x)$ as $\mc{A}_q^\ga(g|h)$ is increasing in $h$, we can rewrite this as
	\begin{equation*}
		\gm\{x \in B : h_{g,q,M}^\ga(x) \geq r\} \geq (1-M^{-q}C_{X,\ga})\gm(B).
	\end{equation*}
	Choosing $M > C_{X,\ga}^{1/q}$ completes the proof.
\end{proof}

\begin{cor}\label{stopfuncor}
	With $X$, $g$, $q$, and $\ga$ as in the statement of the previous lemma, there exists $M = M(X,q,\ga)$ such that whenever $\gF$ is a nonnegative measurable function on $X^+$, we have
	\begin{equation*}
		\iint_{X^+} \gF(y,t) V(y,\ga t) \, d\gm(y) \, dt \lesssim_{X,\ga} \int_X \iint_{\gG_{h_{g,q,M}^\ga(x)/\ga}^\ga(x)} \gF(y,t) \, d\gm(y) \, dt \, d\gm(x).
	\end{equation*}
\end{cor}

\begin{proof}
	This is a straightforward application of Fubini--Tonelli's theorem along with the previous lemma.
	Taking $M$ sufficiently large, Lemma \ref{stopfun} gives
	\begin{align*}
		\iint_{X^+} \gF(y,t) V(y,\ga t) \, d\gm(y) \, dt
		&\lesssim_{X,\ga} \iint_{X^+} \gF(y,t) \int_{\{x \in B(y,\ga t) : h_{g,q,\M}^\ga(x) \geq \ga t\}} \, d\gm(x) \, d\gm(y) \, dt \\
		&= \int_X \int_0^{h_{g,q,M}^\ga(x)/\ga} \int_{B(x,\ga t)} \gF(y,t) \, d\gm(y) \, dt \, d\gm(x) \\
		&= \int_X \iint_{\gG_{h_{g,q,M}^\ga(x)/\ga}^\ga(x)} \gF(y,t) \, d\gm(y) \, dt \, d\gm(x)
	\end{align*}
	as required.
\end{proof}

We are now ready for the proof of the main duality result.

\begin{proof}[Proof of Proposition \ref{duality2}]
	First suppose $f \in T^{1,q,\ga}(X)$ and $g \in T^{\infty,q^\prime,\ga}(X)$.
	By Corollary \ref{stopfuncor}, there exists $M = M(X,q,\ga) > 0$ such that
	\begin{equation*}
		\iint_{X^+} |f(y,t)||g(y,t)| \, d\gm(y) \, \frac{dt}{t} \lesssim_{X,\ga} \int_X \iint_{\gG_{h(x)}^\ga(x)} |f(y,t)||g(y,t)| \, \frac{d\gm(y)}{V(y,\ga t)} \, \frac{dt}{t} \, d\gm(x),
	\end{equation*}
	where $h(x) := h_{g,q^\prime,M}^\ga(x)/\ga$.
	Using H\"older's inequality and the definition of $h(x)$, we find that
	\begin{align*}
		\int_X \left( \iint_{\gG_{h(x)}^\ga(x)} |f(y,t)| |g(y,t)| \frac{d\gm(y)}{V(y,\ga t)} \, \frac{dt}{t} \right) \, d\gm(x)
		&\leq \int_X \mc{A}_q^\ga(f|h(x))(x) \mc{A}_{q^\prime}^\ga(g|h(x))(x) \, d\gm(x) \\
		&\leq M\int_X \mc{A}_q^\ga(f)(x) \mc{C}_{q^\prime}^\ga(g)(x) \, d\gm(x) \\
		&\lesssim_{X,q,\ga} \norm{f}_{T^{1,q,\ga}(X)} \norm{g}_{T^{\infty,q,\ga}(X)}.
	\end{align*}
	Hence every $g \in T^{\infty,q^\prime,\ga}(X)$ induces a bounded linear functional on $T^{1,q,\ga}(X)$ via the pairing $\langle f,g \rangle$ above, and so $T^{\infty,q^\prime,\ga}(X) \subset (T^{1,q,\ga}(X))^*$.
	
	Conversely, suppose $\ell \in (T^{1,q,\ga}(X))^*$.
	Then as in the proof of Proposition \ref{duality1}, from $\ell$ we construct a function $g \in L_\text{loc}^{q^\prime}(X^+)$ such that
	\begin{equation*}
		\iint_{X^+} f(y,t)\overline{g(y,t)} \, d\gm(y) \, \frac{dt}{t} = \ell(f)
	\end{equation*}
	for all $f \in T^{1,q,\ga}(X)$ with cylindrical support.
	We just need to show that $g$ is in $T^{\infty,q^\prime,\ga}(X)$.
	By the definition of the $T^{\infty,q^\prime,\ga}(X)$ norm, it suffices to estimate
	\begin{equation*}
		\left( \frac{1}{\gm(B)} \iint_{T^\ga(B)} |g(y,t)|^{q^\prime} \, d\gm(y) \, \frac{dt}{t} \right)^{\frac{1}{q^\prime}},
	\end{equation*}
	where $B \subset X$ is an arbitrary ball.
	
	For all nonnegative $\gy \in L^q(T^\ga(B))$ with $\norm{\gy}_{L^q(T^\ga(B))} \leq 1$, using that $S^\ga(T^\ga(B)) = B$ we have that
	\begin{align*}
		\norm{\gy}_{T^{1,q,\ga}(X)}
		&= \int_B \mc{A}_q^\ga(\gy)(x) \, d\gm(x) \\
		&\leq \gm(B)^{1/q^\prime} \norm{\gy}_{T^{q,q,\ga}(X)} \\
		&= \gm(B)^{1/q^\prime} \norm{\gy}_{L^q(X^+)} \\
		&\leq \gm(B)^{1/q^\prime}.
	\end{align*}
	In particular, $\gy$ is in $T^{1,q,\ga}(X)$, so we can write
	\begin{equation*}
		\iint_{T^\ga(B)} \overline{g}\gy \, d\gm \frac{dt}{t} = \ell(\gy).
	\end{equation*}
	Arguing by duality and using the above computation, we then have
	\begin{align*}
		\left( \frac{1}{\gm(B)} \iint_{T^\ga(B)} |g(y,t)|^{q^\prime} \, d\gm(y) \, \frac{dt}{t} \right)^{1/q^\prime}
		&= \gm(B)^{-1/q^\prime} \sup_\gy \iint_{T^\ga(B)} \overline{g}\gy \, d\gm \frac{dt}{t} \\
		&= \gm(B)^{-1/q^\prime} \sup_\gy \ell(\gy) \\
		&\leq \gm(B)^{-1/q^\prime} \norm{\ell} \norm{\gy}_{T^{1,q,\ga}(X)} \\
		&\leq \norm{\ell},
	\end{align*}
	where the supremum is taken over all $\gy$ described above.
	Now taking the supremum over all balls $B \subset X$, we find that
	\begin{equation*}
		\norm{g}_{T^{\infty,q^\prime,\ga}(X)} \leq \norm{\ell},
	\end{equation*}
	which completes the proof that $(T^{1,q,\ga}(X))^* \subset T^{\infty,q^\prime,\ga}(X)$.
	\end{proof}

Once Proposition \ref{duality2} is established, we can obtain the full scale of interpolation using the `convex reduction' argument of \cite[Theorem 3]{aB92} and Wolff's reiteration theorem (see \cite{tW82} and \cite{JNP83}).

\begin{prop}\label{interpolation2}
	Suppose that $X$ is doubling.
	Then for $p_0,p_1 \in [1,\infty]$ (not both equal to $\infty$), $q_0$ and $q_1$ in $(1,\infty)$, $\gq \in [0,1]$, and $\ga > 0$, we have (up to equivalence of norms)
	\begin{equation*}
		[T^{p_0,q_0,\ga}(X),T^{p_1,q_1,\ga}(X)]_\gq = T^{p,q,\ga}(X),
	\end{equation*}
	where $1/p = (1-\gq)/p_0 + \gq/p_1$ and $1/q = (1-\gq)/q_0 + \gq/q_1$.
\end{prop}

\begin{proof}
	First we will show that
	\begin{equation}\label{claiminterp}
		[T^{1,q_0,\ga}(X),T^{p_1,q_1,\ga}(X)]_\gq \supset T^{p,q,\ga}(X).
	\end{equation}
	Suppose $f \in T^{p,q,\ga}(X)$ is a cylindrically supported simple function.
	Then there exists another cylindrically supported simple function $g$ such that $f = g^2$.
	Then
	\begin{equation*}
		\norm{f}_{T^{p,q,\ga}(X)} = \norm{g}_{T^{2p,2q,\ga}(X)}^2,
	\end{equation*}
	and so $g$ is in $T^{2p,2q,\ga}(X)$.
	By Proposition \ref{interpolation1} we have the identification
	\begin{equation}\label{ident}
		T^{2p,2q,\ga}(X) = [T^{2,2q_0,\ga}(X),T^{2p_1,2q_1,\ga}(X)]_\gq
	\end{equation}
	up to equivalence of norms, and so by the definition of the complex interpolation functor (see Section \ref{ainterpolation}), there exists for each $\ge > 0$ a function
	\begin{equation*}
		G_\ge \in \mc{F}(T^{2,2q_0,\ga}(X),T^{2p_1,2q_1,\ga}(X))
	\end{equation*}
	such that $G_\ge(\gq) = g$ and
	\begin{align*}
		\norm{G_\ge}_{\mc{F}(T^{2,2q_0,\ga}(X),T^{2p_1,2q_1,\ga}(X)} &\leq (1+\ge)\norm{g}_{[T^{2,2q_0,\ga}(X),T^{2p_1,2q_1,\ga}(X)]_\gq} \\
		&\simeq (1+\ge) \norm{g}_{T^{2p,2q,\ga}(X)},
	\end{align*}
	the implicit constant coming from the norm equivalence \eqref{ident}.
	Define $F_\ge := G_\ge^2$.
	Then we have
	\begin{equation*}
		F_\ge \in \mc{F}(T^{1,q_0,\ga}(X),T^{p_1,q_1,\ga}(X)),
	\end{equation*}
	with
	\begin{align*}
		\norm{F_\ge}_{\mc{F}(T^{1,q_0,\ga}(X),T^{p_1,q_1,\ga}(X))}
		&= \norm{G_\ge}_{\mc{F}(T^{2,2q_0,\ga}(X),T^{2p_1,2q_1,\ga}(X))}^2 \\
		&\lesssim (1+\ge)^2 \norm{g}_{T^{2p,2q,\ga}(X)}^2 \\
		&= (1+\ge)^2 \norm{f}_{T^{p,q,\ga}(X)}.
	\end{align*}
	Therefore
	\begin{equation*}
		\norm{f}_{[T^{1,q_0,\ga}(X),T^{p_1,q_1,\ga}(X)]_\gq} \lesssim \norm{f}_{T^{p,q,\ga}(X)},
	\end{equation*}
	and so the inclusion \eqref{claiminterp} follows from the fact that cylindrically supported simple functions are dense in $T^{p,q,\ga}(X)$.
	
	By the duality theorem \cite[Corollary 4.5.2]{BL76} for interpolation (using that $T^{p_1,q_1,\ga}(X)$ is reflexive, the inclusion \eqref{claiminterp}, and Propositions \ref{duality1} and \ref{duality2}, we have
	\begin{equation*}
		[T^{p_1^\prime,q_1^\prime,\ga}(X), T^{\infty,q_0^\prime,\ga}(X)]_{1-\gq} \subset T^{p^\prime,q^\prime,\ga}(X).
	\end{equation*}
	Therefore we have the containment
	\begin{equation}\label{inftyinterp}
		[T^{p_0,q_0,\ga}(X),T^{\infty,q_1,\ga}(X)]_\gq \subset T^{p,q,\ga}(X).
	\end{equation}
	
        The reverse containment can be obtained from
        \begin{equation}\label{revcontain}
          [T^{1,q_0,\ga}(X),T^{p_1,q_1,\ga}(X)]_\gq \subset T^{p,q,\ga}(X)
        \end{equation}
        (for $p_1$, $q_0$, $q_1 \in (1,\infty)$) by duality.
        The containment \eqref{revcontain} can be obtained as in Remark \ref{refsug}, with $p_0 = 1$ not changing the validity of this method.\footnote{We thank the anonymous referee once more for this suggestion.}
	
	Finally, it remains to consider the case when $p_0 = 1$ and $p_1 = \infty$.
	This is covered by Wolff reiteration.
	Set $A_1 = T^{1,q_0,\ga}(X)$, $A_2 = T^{p,q,\ga}(X)$, $A_3 = T^{p+1,q_3,\ga}(X)$, and $A_4 = T^{\infty,q_1,\ga}(X)$ for an approprate choice of $q_3$.\footnote{More precisely, we need to take $1/q_3 = (1-1/p^\prime)/q_0 + (1/p^\prime)/q_1$.}
	Then for an appropriate index $\gh$, we have $[A_1,A_3]_{\gq/\gh} = A_2$ and $[A_2,A_4]_{(\gh-\gq)/(1-\gq)} = A_3$.
	Therefore by Wolff reiteration, we have $[A_1,A_4]_\gq = A_2$; that is, $[T^{1,q_0,\ga}(X),T^{\infty,q_1,\ga}(X)]_\gq = T^{p,q,\ga}(X)$.
	This completes the proof.
\end{proof}

\begin{rmk}
	Note that doubling is not explicitly used in the above proof; it is only required to the extent that it is needed to prove Propositions \ref{duality1} and \ref{duality2} (as Proposition \ref{interpolation1} follows from \ref{duality1}).
	If these propositions could be proven under some assumptions other than doubling, then it would follow that Proposition \ref{interpolation2} holds under these assumptions.
\end{rmk}

\begin{rmk}\label{cmserror}
	The proof of \cite[Lemma 5]{CMS85}, which amounts to proving the containment \eqref{claiminterp}, contains a mistake which is seemingly irrepairable without resorting to more advanced techniques.
	This mistake appears on page 323, line -3, when it is stated that ``$A(f_k)$ is supported in $O_k^* - O_{k+1}$'' (and in particular, that $A(f_k)$ is supported in $O_{k+1}^c$).
	However (reverting to our notation), since $f_k := \mb{1}_{T((O_k)_\gg^*) \sm T((O_{k+1})_\gg^*)} f$, $\mc{A}_2^1(f_k)$ is supported on
	\begin{equation*}
		S^1(T((O_k)_\gg^*) \sm T((O_{k+1})_\gg^*)) = (O_k)_\gg^*
	\end{equation*}
	and we cannot conclude that $\mc{A}_2^1(f_k)$ is supported away from $O_{k+1}$.
	Simple $1$-dimensional examples can be constructed which show that this is false in general.
	Hence the containment \eqref{claiminterp} is not fully proven in \cite{CMS85}; the first valid proof in the Euclidean case that we know of is  in \cite{aB92} (the full range of interpolation is not obtained in \cite{HTV91}.)
\end{rmk}

\subsection{Change of aperture}

Under the doubling assumption, the change of aperture result can be proven without assuming (NI) by means of the vector-valued method.
The proof is a combination of the techniques of \cite{HTV91} and \cite{aB92}.

\begin{prop}
	Suppose $X$ is doubling.
	For $\ga,\gb \in (0,\infty)$ and $p,q \in (0,\infty)$, the tent space (quasi-)norms $\norm{\cdot}_{T^{p,q,\ga}(X)}$ and $\norm{\cdot}_{T^{p,q,\gb}(X)}$ are equivalent.
\end{prop}

\begin{proof}
	First suppose $p,q \in (1,\infty)$.
	Since $X$ is doubling, we can replace our definition of $\mc{A}_q^\ga$ with the definition
	\begin{equation*}
		\mc{A}_q^\ga(f)(x)^q := \iint_{\gG^\ga(x)} |f(y,t)|^q \, \frac{d\gm(y)}{V(y,t)} \, \frac{dt}{t};
	\end{equation*}
	using the notation of Section \ref{initdefns}, this is the definition with $\mb{a} = y$ and $\mb{b} = 1$.
	Having made this change, the vector-valued approach to tent spaces (see Section \ref{dvvi}) transforms as follows.
	The tent space $T^{p,q,\ga}(X)$ now embeds isometrically into $L^p(X;L_1^q(X^+))$ via the operator $T_\ga$ defined, as before, by
	\begin{equation*}
		T_\ga f(x)(y,t) := f(y,t) \mb{1}_{\gG^\ga(x)}(y,t)
	\end{equation*}
	for $f \in T^{p,q,\ga}(X)$.
	The adjoint of $T_\ga$ is the operator $\gP_\ga$, now defined by
	\begin{equation*}
		(\gP_\ga G)(y,t) := \frac{1}{V(y,t)} \int_{B(y,\ga t)} G(z)(y,t) \, d\gm(z)
	\end{equation*}
	for $G \in L^p(X;L_1^q(X^+))$.
	The composition $P_\ga := T_\ga \gP_\ga$ is then a bounded projection from $L^p(X;L_1^q(X^+))$ onto $T_\ga T^{p,q,\ga}(X)$, and can be written in the form
	\begin{equation*}
		P_\ga G(x)(y,t) = \frac{\mb{1}_{\gG^\ga(x)}(y,t)}{V(y,t)} \int_{B(y,\ga t)} G(z)(y,t) \, d\gm(z).
	\end{equation*}
	For $f \in T^{p,q,\ga}(X)$, we can easily compute
	\begin{equation}\label{PTident}
		P_\gb T_\ga f(x)(y,t) = T_\gb f(x)(y,t) \frac{V(y,\min(\ga,\gb)t)}{V(y,t)}.
	\end{equation}

	Without loss of generality, suppose $\gb > \ga$.
	Then we obviously have
	\begin{equation*}
		\norm{\cdot}_{T^{p,q,\ga}(X)} \lesssim_{q,\ga,\gb,X} \norm{\cdot}_{T^{p,q,\gb}(X)}
	\end{equation*}
	by Remark \ref{obvcoa}.
	It remains to show that
	\begin{equation}\label{coaeq}
		\norm{\cdot}_{T^{p,q,\gb}(X)} \lesssim_{p,q,\ga,\gb,X} \norm{\cdot}_{T^{p,q,\ga}(X)}.
	\end{equation}
	From \eqref{PTident} and doubling, for $f \in T^{p,q,\ga}(X)$ we have that
	\begin{equation*}
		T_\gb f(x)(y,t) \lesssim_{X,\ga} P_\gb T_\ga f(x)(y,t),
	\end{equation*}
	and so we can write
	\begin{align*}
		\norm{f}_{T^{p,q,\gb}(X)}
		&= \norm{T_\gb f}_{L^p(X;L_1^q(X^+))} \\
		&\lesssim_{X,\ga} \norm{P_\gb T_\ga f}_{L^p(X;L_1^q(X^+))} \\
		&\leq \norm{P_\gb}_{\mc{L}(L^p(X;L_1^q(X^+)))} \norm{T_\ga f}_{L^p(X;L_1^q(X^+))} \\
		&\lesssim_{p,q,\gb,X} \norm{f}_{T^{p,q,\ga}(X)} 
	\end{align*}
	since $P_\gb$ is a bounded operator on $L^p(X;L_1^q(X^+))$.
	This shows \eqref{coaeq}, and completes the proof for $p,q \in (1,\infty)$.
	
	Now suppose that at least one of $p$ and $q$ is not in $(1,\infty)$, and suppose $f \in T^{p,q,\ga}(X)$ is a cylindrically supported simple function.
	Choose an integer $M$ such that both $Mp$ and $Mq$ are in $(1,\infty)$.
	Then there exists a cylindrically supported simple function $g$ with $g^M = f$.
	We then have
	\begin{align*}
		\norm{f}_{T^{p,q,\ga}(X)}^{1/M}
		&= \norm{g^M}_{T^{p,q,\ga}(X)}^{1/M} \\
		&= \norm{g}_{T^{Mp,Mq,\ga}(X)} \\
		&\simeq_{p,q,\ga,\gb,X} \norm{g}_{T^{Mp,Mq,\gb}(X)} \\	
		&= \norm{f}_{T^{p,q,\gb}(X)}^{1/M},
	\end{align*}
	and so the result is true for cylindrically supported simple functions, with an implicit constant which does not depend on the support of such a function.
	Since the cylindrically supported simple functions are dense in $T^{p,q,\ga}(X)$, the proof is complete.
\end{proof}

\begin{rmk}
	Written more precisely, with $p,q \in (0,\infty)$ and $\gb < 1$, the inequality \eqref{coaeq} is of the form
	\begin{equation*}
		\norm{\cdot}_{T^{p,q,1}(X)} \lesssim_{p,q,X} \sup_{(y,t) \in X^+} \left( \frac{V(y,t)}{V(y,\gb t)} \right)^M \norm{\cdot}_{T^{p,q,\gb}(X)}.  
	\end{equation*}
	where $M$ is such that $Mp,Mq \in (1,\infty)$.
\end{rmk}

\subsection{Relations between $\mc{A}$ and $\mc{C}$}

Again, this proposition follows from the methods of \cite{CMS85}.

\begin{prop}
		Suppose $X$ satisfies (HL), and suppose $0 < q < p < \infty$ and $\ga > 0$.
		Then
		\begin{equation*}
			\norm{\mc{C}_q^\ga(f)}_{L^p(X)} \lesssim_{p,q,X} \norm{\mc{A}_q^\ga(f)}_{L^p(X)}.
		\end{equation*}
\end{prop}

\begin{proof}
	Let $B \subset X$ be a ball.
	Then by Fubini--Tonelli's theorem, using $S^\ga(T^\ga(B)) = B$,
	\begin{align*}
		\frac{1}{\gm(B)}\iint_{T^\ga(B)} |f(y,t)|^q \, d\gm(y) \, \frac{dt}{t}
		&= \frac{1}{\gm(B)} \iint_{T^\ga(B)} \frac{|f(y,t)|^q}{V(y,\ga t)} \int_{B(y,\ga t)} \, d\gm(x) \, d\gm(y) \, \frac{dt}{t} \\
		&= \frac{1}{\gm(B)} \int_X \iint_{T^\ga(B)} \mb{1}_{B(y,\ga t)}(x) |f(y,t)|^q \, \frac{d\gm(y)}{V(y,\ga t)} \, \frac{dt}{t} \, d\gm(x) \\
		&= \frac{1}{\gm(B)} \int_B \iint_{T^\ga(B)} \mb{1}_{B(x,\ga t)}(y) |f(y,t)|^q  \, \frac{d\gm(y)}{V(y,\ga t)} \, \frac{dt}{t} \, d\gm(x)\\
		&\leq \frac{1}{\gm(B)} \int_B \iint_{X^+} \mb{1}_{B(x,\ga t)}(y) |f(y,t)|^q \, \frac{d\gm(y)}{V(y,\ga t)} \, \frac{dt}{t} \, d\gm(x)\\
		&= \frac{1}{\gm(B)} \int_B \mc{A}_q^\ga(f)(x)^q \, d\gm(x).
	\end{align*}
	Now fix $x \in X$ and take the supremum of both sides of this inequality over all balls $B$ containing $x$.
	We find that
	\begin{equation*}
		\mc{C}_q^\ga(f)(x)^q \leq \mc{M}(\mc{A}_q^\ga(f)^q)(x).
	\end{equation*}
	Since $p/q > 1$, we can apply (HL) to get
	\begin{align*}
		\norm{\mc{C}_q^\ga(f)}_{L^p(X)}
		&\leq \norm{\mc{M}(\mc{A}_q^\ga(f)^q)^{1/q}}_{L^p(X)} \\
		&= \norm{\mc{M}(\mc{A}_q^\ga(f)^q)}_{L^{p/q}(X)}^{1/q} \\
		&\lesssim_{p,q,X} \norm{\mc{A}_q^\ga(f)^q}_{L^{p/q}(X)}^{1/q} \\
		&= \norm{\mc{A}_q^\ga(f)}_{L^p(X)}
	\end{align*}
	as desired.
\end{proof}

\begin{rmk}
	If $X$ is doubling, and if $p,q \in (0,\infty)$, then for $\ga > 0$ we also have that
	\begin{equation*}
		\norm{\mc{A}_q^\ga(f)}_{L^p(X)} \lesssim_{p,q,X} \norm{\mc{C}_q^\ga(f)}_{L^p(X)}.
	\end{equation*}
	This can be proven as in \cite[\textsection 6]{CMS85}, completely analogously to the proofs above.
\end{rmk}

\begin{appendix}

\section{Assorted lemmas and notation}

\subsection{Tents, cones, and shadows}

\begin{lem}\label{tinclusion}
	Suppose $A$ and $B$ are subsets of $X$, with $A$ open, and suppose $T^\ga(A) \subset T^\ga(B)$.
	Then $A \subset B$.
\end{lem}

\begin{proof}
	Suppose $x \in A$.
	Then $\dist(x,A^c) > 0$ since $A$ is open, and so $\dist(x,A^c) > \ga t$ for some $t > 0$.
	Hence $(x,t) \in T^\ga(A) \subset T^\ga(B)$, so that $\dist(x,B^c) > \ga t > 0$.
	Therefore $x \in B$.
\end{proof}

\begin{lem}\label{shadows}
	Let $C \subset X^+$ be cylindrical, and suppose $\ga > 0$.
	Then $S^\ga(C)$ is bounded.
\end{lem}

\begin{proof}
	Write $C \subset B(x,r) \times (a,b)$ for some $x \in X$ and $r,a,b > 0$.
	Then $S^\ga(C) \subset S^\ga(B(x,r) \times (a,b))$, and one can easily show that
	\begin{equation*}
		S^\ga(B(x,r) \times (a,b)) \subset B(x,r+\ga b),
	\end{equation*}
	showing the boundedness of $S^\ga(C)$.
\end{proof}

\begin{lem}\label{incln}
	Let $C \subset X^+$, and suppose $\ga > 0$.
	Then $T^\ga(S^\ga(C))$ is the minimal $\ga$-tent containing $C$, in the sense that $T^\ga(S) \supset C$ for some $S \subset X$ implies that $T^\ga(S^\ga(C)) \subset T^\ga(S)$.
\end{lem}

\begin{proof}
	A straightforward set-theoretic manipulation shows that $C$ is contained in $T^\ga(S^\ga(C))$.
	We need to show that $S^\ga(C)$ is minimal with respect to this property.
	
	Suppose that $S \subset X$ is such that $C \subset T^\ga(S)$, and suppose $(w,t_w)$ is in $T^\ga(S^\ga(C))$.
	With the aim of showing that $\dist(w,S^c) > \ga t_w$, suppose that $y \in S^c$.
	Then $\gG^\ga(y) \cap T^\ga(S) = \varnothing$, and so $\gG^\ga(y) \cap C = \varnothing$ since $T^\ga(S)$ contains $C$.
	Thus $y \in S^\ga(C)^c$, and so
	\begin{equation*}
		d(w,y) \geq \dist(w,S^\ga(C)^c) > \ga t_w
	\end{equation*}
	since $(w,t_w) \in T^\ga(S^\ga(C))$.
	Taking an infimum over $y \in S^c$, we get that
	\begin{equation*}
		\dist(w,S^c) > \ga t_w,
	\end{equation*}
	which says precisely that $(w,t_w)$ is in $T^\ga(S)$.
	Therefore $T^\ga(S^\ga(C)) \subset T^\ga(S)$ as desired.
\end{proof}

\begin{lem}\label{betas}
	For a cylindrical subset $K \subset X^+$, define
	\begin{equation*}
		\gb_0(K) := \inf_{B \subset X} \{\gm(B) : T^\ga(B) \cap K \neq \varnothing\} \quad \text{and} \quad \gb_1(K) := \inf_{B \subset X} \{\gm(B) : T^\ga(B) \supset K\},
	\end{equation*}
	with both infima taken over the set of balls $B$ in $X$.
	Then $\gb_1(K)$ is positive, and if $X$ is proper or doubling, then $\gb_0(K)$ is also positive.
\end{lem}

\begin{proof}
	We first prove that $\gb_0 := \gb_0(K)$ is positive, assuming that $X$ is proper or doubling.
	Write
	\begin{equation*}
		K \subset \overline{C} := \overline{B(x_0,r_0)} \times [a_0,b_0]
	\end{equation*}
	for some $x_0 \in X$ and $a_0,b_0,r_0 > 0$.
	If $B$ is a ball such that $T^\ga(B) \cap K \neq \varnothing$, then we must have $T^\ga(B) \cap \overline{C} \neq \varnothing$, and so we can estimate
	\begin{equation*}
		\gb_0 \geq \inf_{B \subset X}\{\gm(B) : T^\ga(B) \cap \overline{C} \neq \varnothing\}.
	\end{equation*}

	Note that if $B = B(c(B),r(B))$ is a ball with $c(B) \in \overline{B(x_0,r_0)}$, then $T^\ga(B) \cap \overline{C} \neq \varnothing$ if and only if $r(B) \geq \ga a_0$.
	Defining
	\begin{equation*}
		I(x) := \inf\{V(x,r) : r > 0, T^\ga(B(x,r)) \cap \overline{C} \neq \varnothing\}
	\end{equation*}
	for $x \in X$, we thus see that $I(x) = V(x,\ga a_0)$ when $x \in \overline{B(x_0,r_0)}$, and so $I|_{\overline{B(x_0,r_0)}}$ is lower semicontinuous as long as the volume function is lower semicontinuous.
	
	Now suppose $B = B(y,\gr)$ is any ball with $T^\ga(B) \cap \overline{C} \neq \varnothing$.
	Let $(z,t_z)$ be a point in $T^\ga(B) \cap \overline{C}$.
	We claim that the ball
	\begin{equation*}
		\wtd{B} := B\left(z,\frac{1}{2}\left(\gr - d(z,y) + \ga t_z\right)\right)
	\end{equation*}
	is contained in $B$, centred in $\overline{B(x_0,r_0)}$, and is such that $T^\ga(\wtd{B}) \cap \overline{C} \neq \varnothing$.
	The second fact is obvious: $(z,t_z) \in \overline{C}$ implies $z \in \overline{B(x_0,r_0)}$.
	For the first fact, observe that
	\begin{align*}
		\wtd{B} &\subset B(y,d(z,y) + (\gr - d(z,y) + \ga t_z)/2) \\
		&= B(y,(\gr + d(z,y) + \ga t_z)/2) \\
		&\subset B(y,(\gr + (\gr - \ga t_z) + \ga t_z)/2) \\
		&= B(y,\gr),
	\end{align*}
	since $(z,t_z) \in T^\ga(B)$ implies that $d(z,y) < \gr - \ga t_z$.
	Finally, we have $(z,t_z) \in T^\ga(\wtd{B})$: since $c(\wtd{B}) = z$, we just need to show that $t_z < r(\wtd{B})/\ga$.
	Indeed, we have
	\begin{equation*}
		\frac{r(\wtd{B})}{\ga}
		= \frac{1}{2}\left(\frac{\gr - d(z,y)}{\ga} + t_z\right),
	\end{equation*}
	and $t_z < (\gr - d(z,y))/\ga$ as above.
	
	The previous paragraph shows that
	\begin	{equation*}
		\inf_{x \in X}I(x) \geq \inf_{x \in \overline{B(x_0,r_0)}} I(x),
	\end{equation*}
	and so we are reduced to showing that the right hand side of this inequality is positive, since $\gb_0 \geq \inf_{x \in X} I(x)$.
	
	\begin{description}
		\item[If $X$ is proper:]
		Since $\overline{B(x_0,r_0)}$ is compact and $I|_{\overline{B(x_0,r_0)}}$ is lower semicontinuous, $I|_{\overline{B(x_0,r_0)}}$ attains its infimum on $\overline{B(x_0,r_0)}$.
		That is,
		\begin{equation}\label{positivity}
			\inf_{x \in \overline{B(x_0,r_0)}} I(x) = \min_{x \in \overline{B(x_0,r_0)}} I_x > 0,
		\end{equation}
		by positivity of the ball volume function.
		
		\item[If $X$ is doubling:]
		Since $I(x) = V(x,\ga a_0)$ when $x \in \overline{B(x_0,r_0)}$, we can write
		\begin{equation*}
			\inf_{x \in \overline{B(x_0,r_0)}} I(x) \geq \inf_{x \in \overline{B(x_0,r_0)}} V(x,\ge),
		\end{equation*}
		where $\ge = \min(\ga a_0,3r_0)$.
		If $x \in \overline{B(x_0,r_0)}$, then $\overline{B(x_0,r_0)} \subset \overline{B(x,2r_0)} \subset B(x,3r_0)$, and so since $3r_0/\ge \geq 1$,
		\begin{align*}
			V(x_0,r_0) &\leq V(x,3r_0) \\
			&= V(x,\ge(3r_0/\ge)) \\
			&\lesssim_X V(x,\ge).
		\end{align*}
		Hence $V(x,\ge) \gtrsim_X V(x_0,r_0)$, and therefore
		\begin{equation}\label{infimum}
			\inf_{x \in \overline{B(x_0,r_0)}} V(x,\ge) \gtrsim V(x_0,r_0) > 0
		\end{equation}
		as desired.
	\end{description}
	
	We now prove that $\gb_1 = \gb_1(K)$ is positive.
	Recall from Lemma \ref{incln} that if $T^\ga(B) \supset K$, then $T^\ga(B) \supset T^\ga(S^\ga(K))$.
	Since shadows are open, Lemma \ref{tinclusion} tells us that $B \supset S^\ga(K)$.
	Hence $\gm(B) \geq \gm(S^\ga(K))$, and so
	\begin{equation*}
		\gb_1 \geq \gm(S^\ga(K)) > 0
	\end{equation*}
	by positivity of the ball volume function.\footnote{If $S^\ga(K)$ is a ball, then $\gb_1(K) = \gm(S^\ga(K))$.}
\end{proof}

\begin{lem}\label{trunccones}
	Let $B$ be an open ball in $X$ of radius $r$.
	Then for all $x \in B$, the truncated cone $\gG_r^\ga(x)$ is contained in $T^\ga((2\ga+1)B)$.
\end{lem}

\begin{proof}
	Suppose $(y,t) \in \gG_r^\ga(x)$ and $z \in ((2\ga+1)B)^c$, so that $d(y,x) < \ga t < \ga r$ and $d(c(B),z) \geq (2\ga + 1)r$.
	Then by the triangle inequality
	\begin{align*}
		d(y,z) &\geq d(c(B),z) - d(c(B),x) - d(x,y) \\
		&> (2\ga + 1)r - r - \ga r \\
		&= \ga r \\
		&> \ga t,
	\end{align*}
	so that $\dist(y,((2\ga + 1)B)^c) > \ga t$, which yields $(y,t) \in T^\ga((2\ga + 1)B)$.
\end{proof}

\subsection{Measurability}

We assume $(X,d,\gm)$ has the implicit assumptions from Section \ref{assumptions}.

\begin{lem}\label{measurability}
	Let $\ga > 0$, and suppose $\gF$ is a non-negative measurable function on $X^+$.
	Then the function
	\begin{equation*}
		g \colon x \mapsto \iint_{\gG^\ga(x)} \gF(y,t) \, d\gm(y) \, \frac{dt}{t}
	\end{equation*}
	is $\gm$-measurable.
\end{lem}

We present two proofs of this lemma: one uses an abstract measurability result, while the other is elementary (and in fact stronger, proving that $g$ is not only measurable but lower semicontinuous).

\begin{proof}[First proof.]
	By \cite[Theorem 3.1]{lM99}, it suffices to show that the function
	\begin{equation*}
		F(x,(y,t)) := \mb{1}_{B(y,\ga t)}(x) \gF(y,t)
	\end{equation*}
	is measurable on $X \times X^+$.
	For $\ge > 0$, define
	\begin{equation*}
		f_\ge(x,(y,t)) := \frac{\dist(x,\overline{B(y,\ga t)})}{\dist(x,\overline{B(y,\ga t))} + \dist(x,B(y,\ga t + \ge)^c)}.
	\end{equation*}
	Then $f_\ge(x,(y,t))$ is continuous in $x$, and converges pointwise to $\mb{1}_{B(y,\ga t)}(x)$ as $\ge \to 0$.
	Hence
	\begin{equation*}
		F(x,(y,t)) = \lim_{\ge \to 0} f_\ge(x,(y,t))\gF(y,t) =: \lim_{\ge \to 0} F_\ge(x,(y,t)),
	\end{equation*}
	and therefore it suffices to show that each $F_\ge(x,(y,t))$ is measurable on $X \times X^+$.
	Since $F_\ge$ is continuous in $x$ and measurable in $(y,t)$, $F_\ge$ is measurable on $X \times X^+$,\footnote{See \cite[Theorem 1]{kG72}, which tells us that $F_\ge$ is Lusin measurable; this implies Borel measurability on $X \times X^+$.} and the proof is complete.
\end{proof}

\begin{proof}[Second proof.]
	For all $x \in X$ and $\ge > 0$, define the vertically translated cone
	\begin{equation*}
		\gG_\ge^\ga(x) := \{(y,t) \in X^+ : (y,t-\ge) \in \gG^\ga(x)\} \subset \gG^\ga(x).
	\end{equation*}
	If $y \in B(x,\ga \ge)$, then is it easy to show that $\gG_\ge^\ga(x) \subset \gG^\ga(y)$: indeed, if $(z,t) \in \gG_\ge^\ga(x)$, then $d(z,x) < \ga(t-\ge)$, and so
	\begin{equation*}
		d(z,y) \leq d(z,x) + d(x,y) < \ga(t-\ge) + \ga\ge = \ga t.
	\end{equation*}
	
	For all $x \in X$ and $\ge > 0$, define
	\begin{equation*}
		g_\ge(x) := \iint_{\gG_\ge^\ga(x)} \gF(y,t) \, d\gm(y) \, \frac{dt}{t}.
	\end{equation*}
	For each $x \in X$, as $\ge \searrow 0$, we have $g_\ge(x) \nearrow g(x)$ by monotone convergence.
	Fix $\gl > 0$, and suppose that $g(x) > \gl$.
	Then there exists $\ge(x)$ such that $g_{\ge(x)}(x) > \gl$.
	If $y \in B(x,\ga\ge(x))$, then by the previous paragraph we have
	\begin{equation*}
		g(y) \geq g_{\ge(x)}(x) > \gl.
	\end{equation*}
	Therefore $g$ is lower semicontinuous, and thus measurable.
	
\end{proof}

\begin{lem}\label{inftymeas}
	Let $f$ be a measurable function on $X^+$, $q \in (0,\infty)$, and $\ga > 0$.
	Then $\mc{C}_q^\ga(f)$ is lower semicontinuous.
\end{lem}

\begin{proof}
	Let $\gl > 0$, and suppose $x \in X$ is such that $\mc{C}_q^\ga(f)(x) > \gl$.
	Then there exists a ball $B \ni x$ such that
	\begin{equation*}
		\frac{1}{\gm(B)} \iint_{T^\ga(B)} |f(y,t)|^q \, d\gm(y) \, \frac{dt}{t} > \gl^q.
	\end{equation*}
	Hence for any $z \in B$, we have $\mc{C}_q^\ga(f)(z) > \gl$, and so the set $\{x \in X : \mc{C}_q^\ga(f)(x) > \gl\}$ is open.
\end{proof}

\subsection{Interpolation}\label{ainterpolation}

Here we fix some notation involving complex interpolation.

An \emph{interpolation pair} is a pair $(B_0,B_1)$ of complex Banach spaces which admit embeddings into a single complex Hausdorff topological vector space.
To such a pair we can associate the Banach space $B_0 + B_1$, endowed with the norm
\begin{equation*}
	\norm{x}_{B_0 + B_1} := \inf\{\norm{x_0}_{B_0} + \norm{x_1}_{B_1} : x_0 \in B_0, x_1 \in B_1, x=x_0 + x_1\}.
\end{equation*}
We can then consider the space $\mc{F}(B_0,B_1)$ of functions $f$ from the closed strip
\begin{equation*}
	\overline{S} = \{z \in \CC : 0 \leq \Re(z) \leq 1\}
\end{equation*}
into the Banach space $B_0 + B_1$, such that
\begin{itemize}
	\item $f$ is analytic on the interior of $S$ and continuous on $\overline{S}$,
	\item $f(z) \in B_j$ whenever $\Re(z) = j$ ($j \in \{0,1\}$), and
	\item the traces $f_j := f|_{\Re z = j}$ ($j \in \{0,1\}$) are continuous maps into $B_j$ which vanish at infinity.
\end{itemize}
The space $\mc{F}(B_0,B_1)$ is a Banach space when endowed with the norm
\begin{equation*}
	\norm{f}_{\mc{F}(B_0,B_1)} := \max\left(\sup_{\Re z = 0} \norm{f(z)}_{B_0}, \sup_{\Re z = 1} \norm{f(z)}_{B_1}\right).
\end{equation*}
We define the \emph{complex interpolation space} $[B_0,B_1]_\gq$ for $\gq \in [0,1]$ to be the subspace of $B_0 + B_1$ defined by
\begin{equation*}
	[B_0,B_1]_\gq := \{f(\gq) : f \in \mc{F}(B_0,B_1)\}
\end{equation*}
endowed with the norm
\begin{equation*}
	\norm{x}_{[B_0,B_1]_\gq} := \inf_{f(\gq) = x} \norm{f}_{\mc{F}(B_0,B_1)}.
\end{equation*}

\end{appendix}
\footnotesize
\bibliographystyle{amsplain}
\bibliography{analysis2}
  
\end{document}